\documentclass[dvipsnames]{amsart}
\usepackage{pict2e}
\usepackage{amsmath,amssymb,amsthm}
\usepackage{extarrows}
\usepackage{fullpage}
\usepackage{todonotes}
 
\allowdisplaybreaks
\usepackage{soul}

\usepackage{color}
\usepackage{hyperref}
\hypersetup{
	colorlinks=true,
	linkcolor=blue,
	citecolor=blue,
	urlcolor=blue
}

\usepackage{graphicx}
\usepackage{mathrsfs}
\usepackage{tikz-cd}
\usepackage{array}
\usepackage{multirow}
\usepackage{soul}
\usepackage{quiver}

\usepackage{xcolor}

\usepackage{tikz}

\newtheorem{theoremAlph}{Theorem}

\newtheorem{theorem}{Theorem}[section]
\newtheorem{lemma}[theorem]{Lemma}	
\newtheorem{proposition}[theorem]{Proposition}
\newtheorem{corollary}[theorem]{Corollary}
\theoremstyle{definition}
\newtheorem{definition}[theorem]{Definition} 
\newtheorem{remark}[theorem]{Remark}

\newcommand{\longdownarrow}{\lower 1.4ex\hbox{\begin{picture}(18,18)(0,0)
\thicklines
\put(0,18){\vector(0,-1){18}}
\end{picture}}}
\newcommand{\longsearrow}{\lower 1.4ex\hbox{\begin{picture}(18,18)(0,0)
\thicklines
\put(0,18){\vector(1,-1){18}}
\end{picture}}}
\newcommand{\longssearrow}{\lower 1.4ex\hbox{\begin{picture}(18,18)(0,0)
\thicklines
\put(0,18){\vector(1,-2){18}}
\end{picture}}}
\newcommand{\longeearrow}{\lower 1.4ex\hbox{\begin{picture}(0,0)(9,9)
\thicklines
\put(0,18){\vector(1,0){18}}
\end{picture}}}
 \usepackage{relsize}
\theoremstyle{definition}

\numberwithin{equation}{section}
\newcommand{\C}{\mathcal{C}}

\title[Spectral Complexes and Currents on the Engel Group]{Spectral Complexes and Currents on the Engel Group} 
\keywords{Pansu pullback, Carnot groups,  Rumin complex, spectral sequences}
 
\subjclass{46L87, 53C17, 22E25, 30L99, 18G40}
\author[F.~Lo Biundo]{Filippa Lo Biundo}
\author[F.~Tripaldi]{Francesca Tripaldi}

\address[F. Lo Biundo]%
{Department of Pure Mathematics, University of Leeds, Woodhouse, LS2 9JT Leeds, UK} 
\email{f.lobiundo@leeds.ac.uk}

\address[F. Tripaldi]%
{Department of Pure Mathematics, University of Leeds, Woodhouse, LS2 9JT Leeds, UK} 
\email{f.tripaldi@leeds.ac.uk}

\date{}

\begin{document}
\maketitle

\begin{abstract}
We study the spectral complexes arising from the de Rham complex on the Engel group $G$ viewed as a truncated multicomplex. We describe the two resulting complexes, analyse their compactly supported counterparts, and define the corresponding spectral currents by duality. We prove that the smooth spectral complexes embed naturally into the current complexes. Finally, we show that the Pansu pullback of a $W_{\mathrm{loc}}^{1,q}$ map $\varphi\colon G\to G$ with $q>7$ induces a morphism into the corresponding complexes of spectral currents.
\end{abstract}

\section{Introduction}
The de Rham complex admits a natural extension from smooth differential
forms to currents. A current is a continuous linear functional on the
space of compactly supported smooth forms, endowed with its usual LF
topology. The exterior differential then extends by transposition, and every
smooth form defines a current by integration. After reversing degrees in
the usual way, integration gives an injective morphism from the
de Rham complex to the corresponding complex of currents. 

On a Carnot group, the stratification of the Lie algebra provides
additional structure that is not visible from the ordinary grading by
form degree. Differential forms decompose according to homogeneous
weight, and the exterior differential splits into components that raise
the weight by different amounts. In this way, the de Rham complex becomes
a truncated multicomplex \cite{lerario2023multicomplexes}, and its
associated spectral sequence separates the different homogeneous orders
of the exterior differential. The first page can be identified with the
space of Rumin forms, while the differentials on the successive pages are
related to the homogeneous components of the Rumin differential. We refer
to \cite{rumin_grenoble,F+T1} for constructions of the Rumin complex on
homogeneous nilpotent Lie groups.

The spectral complexes introduced in
\cite{tripaldi2026spectralcomplexestruncatedmulticomplexes} reorganise
these page differentials into a family of complexes. If
$Z_r^{p,q}$ and $B_r^{p,q}$ denote the cycles and boundaries associated
with the $r^{th}$ page of the spectral sequence of a truncated
multicomplex, the terms of a spectral complex are mixed quotients
\[
E_{r,l}^{p,q}:=\frac{Z_r^{p,q}}{B_l^{p,q}}\,,
\]
where the two indices need not coincide. The differentials connect
successive mixed quotients according to the homogeneous orders occurring
in the Rumin differential. Thus, rather than retaining a single page of
the spectral sequence, a spectral complex records in one cochain complex
the different homogeneous contributions that are relevant in consecutive
degrees.

Rumin currents on Carnot groups can likewise be defined by duality
against compactly supported Rumin forms, with the boundary operator
obtained by transposing the Rumin differential \cite{FSSC6,Can21jga,JuliaPansu2023} (see also
\cite{kleiner2021rumin} for the corresponding
distributional Rumin complex on Heisenberg groups). For spectral
complexes, however, the passage to distributions is not automatic. Their
terms are mixed quotients rather than merely spaces of sections of fixed
homogeneous subbundles. One must therefore specify the topology of the
compactly supported spectral test spaces, account for the auxiliary forms
entering the definitions of the cycles and boundaries, and verify that
integration descends to the complementary quotients.

The aim of the present paper is to develop a distributional theory for
these spectral complexes in a concrete higher-step model. We
work on the Engel group $G$, the connected, simply connected Carnot group
whose Lie algebra has an adapted basis $X_1,X_2,X_3,X_4$ satisfying
\[
[X_1,X_2]=X_3\ ,
\
[X_1,X_3]=X_4\,.
\]
Its topological and homogeneous dimensions are respectively $n=4$ and
$Q=7$. The Engel group has nilpotency step $3$, and the Rumin differential
on $1$-forms contains non-zero homogeneous components that increase the
weight by $2$ and $3$, hence producing two distinct spectral
complexes. We compute all their terms and differentials explicitly in
terms of the left-invariant frame and the homogeneous components of the
Rumin complex.

A central point in passing to currents is the treatment of compactly
supported smooth forms. Membership in $Z_r^{p,q}$ or $B_r^{p,q}$ is defined through a
system of equations involving auxiliary forms. When the construction is
performed in the compactly supported de Rham complex, all these
forms must also be compactly supported. The resulting spaces cannot in
general be obtained by first constructing $Z_r^{p,q}$ and $B_r^{p,q}$ in
the unrestricted smooth complex and then taking their set-theoretic
intersection with compactly supported forms.

For the Engel group, we make this distinction explicit. In one of the
relevant bidegrees, unrestricted smooth coefficients satisfy
\[
B_3^{6,-3}(G)=Z_2^{6,-3}(G)\,,
\]
whereas in another bidegree, the compact-support condition gives
the exact description
\[
Z_3^{1,0}(G)\cap\Omega_c^1(G)
=
\left\{
X_1f\,\theta_1+X_2f\,\theta_2\mid
f\in\C_c^\infty(G)
\right\}
=d_1\bigl(\C_c^\infty(G)\bigr)\,.
\]
These identities are essential for understanding the complementary
pairings used in the current construction.

We next introduce spaces of compactly supported spectral test forms. For
every admissible choice of indices, we set
\[
\mathscr D_{r,l}^{p,k-p}(G)
:=
\frac{Z_r^{p,k-p}(G)\cap\Omega_c^k(G)}
{B_l^{p,k-p}(G)\cap\Omega_c^k(G)},
\]
where the numerator and denominator are constructed entirely inside the
compactly supported de Rham complex. On each fixed compact set, the
relevant cycle space is a closed subspace of a Fr\'echet space of smooth
forms. We then endow $Z_r^{p,k-p}$ with the locally convex topology
induced by the usual LF topology on $\Omega_c^k(G)$ and give
$\mathscr D_{r,l}^{p,k-p}(G)$ the corresponding quotient topology. The
quotient need not be Hausdorff, however its continuous dual is
canonically the annihilator of the boundary space and depends only on its
closure.

Motivated by the complementary degree and weight in the integration
pairing, we define the space of $(r,l)$-spectral currents of bidegree
$(p,k-p)$ by
\[
\mathscr C_{r,l}^{p,k-p}(G)
:=
\left(
\mathscr D_{l,r}^{Q-p,\,n-k-(Q-p)}(G)
\right)'\,.
\]
The spectral differentials on test forms are continuous and hence induce boundary operators on spectral currents. In this way,
the two spectral complexes of the Engel group give rise to two complexes
of spectral currents.

We also identify all the spaces appearing in these current complexes.
They have the expected distributional form: smooth coefficients are
replaced by distributions, while the differential constraints and
quotient relations defining the spectral spaces persist in the
distributional sense. For example, one of the current spaces is naturally
identified with
\[
\left\{
T_1\theta_1+T_2\theta_2\mid
(X_2X_1-X_3)T_2=X_2^2T_1
\right\},
\]
where $T_1,T_2\in\mathcal D'(G)$ and the equation is understood
distributionally. Another term is a quotient of distributional
$3$-forms by the subspace determined by the relation
$X_2T_1=X_1T_2$.

The main result of the paper is the spectral counterpart of the classical
embedding of smooth forms into currents.

\begin{theoremAlph}\label{thm: introduction spectral embedding}
For every admissible choice of $p,k,r,l$, the integration pairing
\[
E_{r,l}^{p,k-p}(G)
\times
\mathscr D_{l,r}^{Q-p,\,n-k-(Q-p)}(G)
\longrightarrow\mathbb R,
\qquad
([\alpha],[\eta])
\longmapsto
\int_G\alpha\wedge\eta,
\]
is well defined and non-degenerate. Moreover, the map
\[
\psi_{r,l}^{p,k-p}\colon
E_{r,l}^{p,k-p}(G)
\longrightarrow
\mathscr C_{r,l}^{p,k-p}(G),
\qquad
\psi_{r,l}^{p,k-p}([\alpha])([\eta])
:=
\int_G\alpha\wedge\eta,
\]
is a well-defined injective linear map. These maps intertwine the spectral
differentials with the transposed boundary operators and therefore define
embeddings of the two spectral complexes of the Engel group into the
corresponding complexes of spectral currents.
\end{theoremAlph}

The paper is organised as follows. In Section~\ref{section 1}, we realise
the de Rham complex of the Engel group as a truncated $3$-multicomplex,
recall its relation with the Rumin complex, and construct the two
associated spectral complexes. We then compute their terms explicitly,
both for unrestricted smooth forms and under the compact-support
convention required for duality. In the final section, we introduce the
locally convex spaces of compactly supported spectral test forms, define
spectral currents and their boundary operators, describe the two current
complexes in terms of distributions, and prove
Theorem~\ref{thm: introduction spectral embedding}.

\section{Spectral complexes on the Engel group}\label{section 1}

The properties presented in the following subsection are meant to apply to the specific case of the Engel group, even though the same (or analogous) statements are known to hold true in the more general case when $G$ is taken to be an arbitrary Carnot group \cite{lerario2023multicomplexes}.
\subsection{The Engel group as a truncated multicomplex}

The specific group that we are going to focus on throughout this paper is the so-called Engel group. This is the connected, simply-connected, 4-dimensional filiform nilpotent Lie group $G$ whose Lie algebra $\mathfrak{g}$ has only two non-zero brackets 
\begin{align*}
    [X_1,X_2]=X_3\ ,\ [X_1,X_3]=X_4\,.
\end{align*}
Its Lie algebra $\mathfrak{g}$ admits several non-equivalent homogeneous structures \cite{hakavuori2022gradings}, one of which is the stratification
\begin{align}\label{eq: stratification}
    \mathfrak{g}=V_1\oplus V_2\oplus V_3\ \text{ with }V_1=\operatorname{span}_\mathbb{R}\{X_1,X_2\}\ ,\ V_2=\operatorname{span}_{\mathbb{R}}\{X_3\}\text{ and }V_3=\operatorname{span}_{\mathbb{R}}\{X_4\}\,.
\end{align}

Using the stratification \eqref{eq: stratification}, we can equip the Engel group with a left-invariant subRiemannian structure by fixing a scalar product on $V_1$, hence obtaining a Carnot group.

Equivalently, the existence of the stratification \eqref{eq: stratification} can be expressed by the fact that $\mathfrak{g}$ admits a family of Lie algebra automorphisms $\delta_\lambda$ with $\lambda>0$ of the form $\delta_\lambda=\operatorname{Exp}(A\ln\lambda)$, where $(X_1,X_2,X_3,X_4)$ is a basis of eigenvectors of $A$. In this basis,  $\delta_\lambda$ is represented by the diagonal matrix $\operatorname{Mat}(\delta_\lambda)=\operatorname{diag}(\lambda,\lambda,\lambda^2,\lambda^3)$, so that
$$\delta_\lambda X_1=\lambda X_1\ ,\
\delta_\lambda X_2=\lambda X_2\ ,\
\delta_\lambda X_3=\lambda^2X_3\ ,\
\delta_\lambda X_4=\lambda^3X_4\,.$$
The automorphisms $\delta_\lambda\colon\mathfrak{g}\to\mathfrak{g}$ are called dilations, while the eigenvalues $(1,1,2,3)$ of $A$ are referred to as their weights.
Since $G$ is connected and simply-connected, each Lie algebra automorphism $\delta_\lambda\colon\mathfrak{g}\to\mathfrak{g}$ integrates uniquely to a Lie group automorphism $\delta_\lambda\colon G\to G$ called a group dilation.

To extend the notion of weight to arbitrary smooth differential forms while regarding smooth functions as coefficients of weight zero, one can use left translations. For every $x\in G$, define
\begin{align*}
    \delta_{\lambda,x}:=d_eL_x\circ\delta_\lambda\circ d_xL_{x^{-1}}
\colon T_xG\to T_xG
\end{align*}

For each $\lambda>0$, the family $(\delta_{\lambda,x})_{x\in G}$ defines a vector-bundle automorphism of $TG$ covering the identity map of $G$, which we again denote by $\delta_\lambda$. It induces a $\C^\infty(G)$-linear action
$\delta_\lambda\colon\Omega^k(G)\to\Omega^k(G)
$
defined as
$$(\delta_\lambda\alpha)_x(v_1,\ldots,v_k)
:=
\alpha_x\bigl(
\delta_{\lambda,x}v_1,\ldots,\delta_{\lambda,x}v_k
\bigr)\ \text{ for every }\alpha\in\Omega^k(G),\,v_1,\ldots,v_k\in T_xG\,,\text{ and }x\in G\,.$$
In this way, the action of $\delta_\lambda$ on a differential form is a fibrewise \(\C^\infty(G)\)-linear action rather than the pullback by the group dilation. In particular, $\delta_\lambda(f\alpha)=f\delta_\lambda\alpha$
for every \(f\in \C^\infty(G)\) and \(\alpha\in\Omega^\bullet(G)\).

It is standard to keep the same notation $\delta_\lambda$ for all these dilation maps. One can check that $\delta_\lambda$ also respects the wedge product, so that
\begin{align*}
    \delta_\lambda(\alpha_1\wedge\alpha_2)=(\delta_\lambda\alpha_1)\wedge(\delta_\lambda\alpha_2)\ \text{ for all }\alpha_1,\alpha_2\in\Omega^\bullet(G)\,.
\end{align*}
\begin{remark}\label{rmk: smooth modules of forms}
    One can consider the subcomplex of the de Rham complex consisting of left-invariant differential forms $\Omega_L^\bullet(G)$. A left-invariant $k$-form is uniquely determined by its value at the identity, where it defines a linear map $\bigwedge^k\mathfrak{g}\to\mathbb{R}$, by identifying the tangent space at the identity with the Lie algebra $\mathfrak{g}$, so that $\Omega_L^k(G)\cong\bigwedge^k\mathfrak{g}^\ast$.

Furthermore, one can identify the tangent space $T_xG$ to $G$ at a point $x\in G$ with $\mathfrak{g}$ by means of the isomorphism $dL_x$, where $L_x$ denotes the left-translation by $x\in G$. For $\xi\in\bigwedge^k\mathfrak{g}^\ast$ and $f\in \mathcal{C}^\infty(G)$, we can regard $f\otimes\xi$ as a smooth $k$-form by 
$$\Phi(f\otimes\xi)_x(v_1,\ldots,v_k)=f(x)\,\xi(d_xL_{x^{-1}}v_1,\ldots,d_xL_{x^{-1}}v_k)\,,$$
giving rise to the isomorphism
\begin{align*}
    \operatorname{Hom}_\mathbb{R}\big(\bigwedge\nolimits^k\mathfrak{g},\mathcal{C}^\infty(G)\big)\cong\mathcal{C}^\infty(G)\otimes\bigwedge\nolimits^k\mathfrak{g}^\ast\cong\Gamma(\bigwedge\nolimits^kT^\ast G)=\Omega^k(G)\,.
\end{align*}
Similarly, when dealing with compactly supported smooth forms $\Omega^\bullet_c(G)$, we can again use left-translation to obtain the isomorphism $\Omega^k_c(G)\cong\mathcal{C}_c^\infty(G)\otimes\bigwedge^k\mathfrak{g}^\ast$.
\end{remark}
\begin{definition}[Weights of forms]\label{def: weights of forms}  We say that a non-zero form $\alpha\in\Omega^\bullet(G)$ is \textit{homogeneous} if there exists $p\in\mathbb{R}$ such that $\delta_\lambda\alpha=\lambda^p\alpha$ for all $\lambda>0$. The number $p$ is referred to as the \textit{weight} of $\alpha$ and we will write it as $w(\alpha)=p$. 
\end{definition}
Given the basis $(X_1,X_2,X_3,X_4)$ adapted to the stratification \eqref{eq: stratification}, we have that
\begin{align*}
    w(X_1)=w(X_2)=1\ ,\ w(X_3)=2 \text{ and }w(X_4)=3\,.
\end{align*}
Its dual basis $(\theta_1,\theta_2,\theta_3,\theta_4)$ such that $\theta_i(X_j)=\delta_{ij}$ also reflects the stratification in terms of weights, as
\begin{align*}
    w(\theta_1)=w(\theta_2)=1\ ,\ w(\theta_3)=2\text{ and }w(\theta_4)=3\,.
\end{align*}

By Definition \ref{def: weights of forms}, we get that non-zero 0-forms $\Omega^0(G)=\mathcal{C}^\infty(G)$ have weight 0, while the volume covector $\operatorname{vol}=\theta_1\wedge\theta_2\wedge\theta_3\wedge\theta_4$ has weight $Q=7$, where $Q$ is the homogeneous dimension of $G$, and hence the Hausdorff dimension associated with its Carnot-Carath\'eodory metric.
Furthermore, the identification $\Omega^\bullet(G)\cong\mathcal{C}^\infty(G)\otimes\bigwedge^\bullet\mathfrak{g}^\ast$ presented in Remark \ref{rmk: smooth modules of forms} implies that a form $\alpha\in\Omega^k(G)$ is homogeneous of weight $p$ if $\alpha=\sum_jf_j\otimes\xi_j$, where $f_j\in\mathcal{C}^\infty(G)$ and $w(\xi_j)=p$ for each $j$. 

\begin{lemma}[Lemma 2.5 in \cite{biundo2026pansupullbackspectralcomplexes}]\label{lem: red lemma} Given two arbitrary homogeneous $k$-forms $\alpha_1,\alpha_2\in\Omega^k(G)$, if $w(\alpha_1)\neq w(\alpha_2)$ then they are linearly independent. 
\end{lemma}
As a consequence, the space of smooth forms $\Omega^k(G)$ admits a direct sum decomposition given by the weight. Throughout this paper, we will express this decomposition using the following notation 
\begin{align}\label{eq: smooth forms according to weight and degree}
    \Omega^k(G)=\bigoplus_{\substack{p+q=k\\0\le p\le 7}}\mathcal{C}^\infty(G)\otimes\bigwedge\nolimits^{p,q}\mathfrak{g}^\ast=\bigoplus_{\substack{p+q=k\\ 0\le p\le 7}}\Omega^{p,q}(G)
\end{align}
and likewise in the case of compactly supported smooth forms
\begin{align*}
    \Omega^k_c(G)=\bigoplus_{\substack{p+q=k\\ 0\le p\le 7}}\mathcal{C}_c^\infty(G)\otimes\bigwedge\nolimits^{p,q}\mathfrak{g}^\ast=\bigoplus_{\substack{p+q=k\\ 0\le p\le 7}}\Omega_c^{p,q}(G)\,.
\end{align*}
Here we are using the notation
$$\bigwedge\nolimits^{p,q}\mathfrak{g}^\ast:=\{\xi\in\bigwedge\nolimits^{p+q}\mathfrak{g}^\ast\mid\delta_\lambda\xi=\lambda^p\xi\text{ for every }\lambda>0\}$$
and so $\bigwedge^{p,q}\mathfrak{g}^\ast$, $\Omega^{p,q}(G)$ and $\Omega^{p,q}_c(G)$ denote the spaces of forms of weight $p$  and degree $p+q=k$ which are respectively left-invariant, smooth and smooth with compact support. 

In the specific case of the Engel group, this direct sum decomposition can be explicitly seen as follows:
\begin{itemize}
    \item $k=1$ $$\bigwedge\nolimits^1\mathfrak{g}^\ast=\mathfrak{g}^\ast=\bigoplus_{1\le p\le 3}\bigwedge\nolimits^{p,1-p}\mathfrak{g}^\ast$$
    where $\bigwedge^{1,0}\mathfrak{g}^\ast=\operatorname{span}_\mathbb{R}\{\theta_1,\theta_2\}$, $\bigwedge^{2,-1}\mathfrak{g}^\ast=\operatorname{span}_\mathbb{R}\{\theta_3\}$, and $\bigwedge^{3,-2}\mathfrak{g}^\ast=\operatorname{span}_\mathbb{R}\{\theta_4\}$;
    \item $k=2$
    \begin{align*}
        \bigwedge\nolimits^{2}\mathfrak{g}^\ast=\bigoplus_{2\le p\le 5}\bigwedge\nolimits^{p,2-p}\mathfrak{g}^\ast
    \end{align*}
where $\bigwedge^{2,0}\mathfrak{g}^\ast=\operatorname{span}_\mathbb{R}\{\theta_1\wedge\theta_2\}$, $\bigwedge^{3,-1}\mathfrak{g}^\ast=\operatorname{span}_\mathbb{R}\{\theta_1\wedge\theta_3,\theta_2\wedge\theta_3\}$, $\bigwedge^{4,-2}\mathfrak{g}^\ast=\operatorname{span}_\mathbb{R}\{\theta_1\wedge\theta_4,\theta_2\wedge\theta_4\}$, and $\bigwedge^{5,-3}\mathfrak{g}^\ast=\operatorname{span}_\mathbb{R}\{\theta_3\wedge\theta_4\}$;
\item $k=3$
$$\bigwedge\nolimits^3\mathfrak{g}^\ast=\bigoplus_{4\le p\le 6}\bigwedge\nolimits^{p,3-p}\mathfrak{g}^\ast$$
where $\bigwedge^{4,-1}\mathfrak{g}^\ast=\operatorname{span}_\mathbb{R}\{\theta_1\wedge\theta_2\wedge\theta_3\}$, $\bigwedge^{5,-2}\mathfrak{g}^\ast=\operatorname{span}_{\mathbb{R}}\{\theta_1\wedge\theta_2\wedge\theta_4\}$, and $\bigwedge^{6,-3}\mathfrak{g}^\ast=\operatorname{span}_\mathbb{R}\{\theta_1\wedge\theta_3\wedge\theta_4,\theta_2\wedge\theta_3\wedge\theta_4\}$;
\item $k=4$
$$\bigwedge\nolimits^4\mathfrak{g}^\ast=\bigwedge\nolimits^{7,-3}\mathfrak{g}^\ast$$
where $\bigwedge^{7,-3}\mathfrak{g}^\ast=\operatorname{span}_\mathbb{R}\{\theta_1\wedge\theta_2\wedge\theta_3\wedge\theta_4\}$.

\end{itemize}
\begin{remark}\label{rmk: weight of non hom form} If a $k$-form $\alpha$ is not homogeneous but instead decomposes as a sum of homogeneous components of different weight, we adopt the same convention as in \cite{kleiner2020pansu,biundo2026pansupullbackspectralcomplexes}, namely
\begin{align*}
    \text{ if }\alpha=\sum_{j=p_{\min} }^{p_{\max}}\alpha_j\ \text{ where }\alpha_j\in\Omega^{j,k-j}(G)\ \text{ then }\ w(\alpha)=p_{\min}\,.
\end{align*}
In other words, the weight of $\alpha$ is defined as the minimal weight among its non-vanishing homogeneous components.    
\end{remark}
\begin{lemma}\label{lem: Lemma 2.7}The stratification \eqref{eq: stratification} of the Lie algebra $\mathfrak{g}$ induces a decomposition of the exterior de Rham differential $d$ acting on the space of smooth forms on the Engel group which can be easily expressed in terms of the weight increase. Given an arbitrary form $\alpha\in\Omega^{p,k-p}(G)$ of homogeneous weight $p$ and degree $k$, one can write
\begin{align*}
    d\alpha=d_0\alpha+d_1\alpha+d_2\alpha+d_3\alpha
\end{align*}
where $d_i\alpha\in\Omega^{p+i,k+1-p-i}(G)$ for each $i=0,1,2,3$.
\end{lemma}
\begin{proof}
    Given an arbitrary smooth form $\alpha\in\Omega^{p,k-p}(G)$ of the form $\alpha=\sum_jf_j\xi_j$ and using the basis $(X_1,X_2,X_3,X_4)$ adapted to the stratification \eqref{eq: stratification}, we get the formula
    \begin{align*}
        d\big(\sum_jf_j\xi_j\big)=&\sum_jdf_j\wedge\xi_j+\sum_jf_j d\xi_j=\sum_j\big(\sum_{i=1}^4X_if_j\theta_i\big)\wedge\xi_j+\sum_jf_j d\xi_j\\=&\underbrace{\sum_{j}\big(X_1f_j\theta_1\wedge\xi_j+X_2f_j\theta_2\wedge\xi_j\big)}_{\text{weight }p+1}+\underbrace{\sum_jX_3f_j\theta_3\wedge\xi_j}_{\text{weight }p+2}+\underbrace{\sum_jX_4f_j\theta_4\wedge\xi_j}_{\text{weight }p+3}+\underbrace{\sum_jf_jd\xi_j}_{\text{weight }p}\\=&d_1\alpha+d_2\alpha+d_3\alpha+d_0\alpha\,.
    \end{align*}
    Notice that, while the increase in weight of the first summand is a clear consequence of the Definition \ref{def: weights of forms}, the fact that $d\xi_j$ either vanishes or is homogeneous of the same weight $p$ as $\xi_j$ for each $j$ is not as straightforward. When dealing with a stratification such as in this case, one can show this using the group's dilations. Indeed, this second summand coincides with the action of the exterior derivative on left-invariant forms, which can be seen as
    \begin{align}\label{formula esplicita d0}
        d\xi_j(Y_1,\ldots,Y_{k+1})=\sum_{1\le i<l\le k+1} (-1)^{i+l}\xi_j([Y_i,Y_l],Y_1,\ldots,\hat{Y}_i,\ldots,\hat{Y}_l,\ldots,Y_{k+1})\ \text{ for all }Y_m\in\mathfrak{g}\,.
    \end{align}
    This formula, combined with the fact that dilations $\delta_\lambda$ are automorphisms of the Lie algebra $\mathfrak{g}$, readily implies that $d$ applied to $\bigwedge^k\mathfrak{g}^\ast$ commutes with the dilations. More explicitly, the fact that $\delta_\lambda(d\xi_j)=d(\delta_\lambda\xi_j)=d(\lambda^p\xi_j)=\lambda^pd\xi_j$ implies that $d\xi_j$ either vanishes or is homogeneous of weight $p$, so that
$$d_0\alpha=\sum_jf_jd\xi_j\in\Omega^{p,k+1-p}(G)\,.$$
\end{proof}
Before applying the spectral sequence machinery developed in \cite{tripaldi2026spectralcomplexestruncatedmulticomplexes}, we first need to show that the de Rham complex $(\Omega^\bullet(G),d)$ on the Engel group is a truncated multicomplex.
\begin{definition}[Adapted from Definition 2.1 in  \cite{lerario2023multicomplexes}] Let $\mathbb K$ be a commutative unital ground ring and let $s\in\mathbb N$. An $s$-multicomplex (also known as a twisted chain complex) is a $(\mathbb{Z},\mathbb{Z})$-graded $\mathbb K$-module $\mathcal{C}$ equipped with maps $d_i\colon\mathcal{C}\to\mathcal{C}$ for $i\ge 0$ of bidegree $\vert d_i\vert=(i,1-i)$ such that\begin{align}\label{eq: structure maps multicomplex}
    \sum_{i+j=n}d_id_j=0\ \text{ for all }n\ge 0\text{ and }d_i=0\text{ for all }i>s\,.
\end{align}
Here, we are using the cohomological convention for the bidegree of the maps $d_i$ adapted from Definition 2.1 in \cite{livernet2020spectral}.
\end{definition}
The following result is the specialisation of Proposition 2.9 in \cite{biundo2026pansupullbackspectralcomplexes} in the case of the Engel group.
\begin{proposition}\label{prop: truncated multicomplex}
    The de Rham complex $(\Omega^\bullet(G),d)$ on the Engel group with stratification \eqref{eq: stratification} is a truncated $3$-multicomplex.
\end{proposition}
\begin{proof}
    The direct sum decomposition according to weights given in \eqref{eq: smooth forms according to weight and degree} endows the space of smooth forms $\Omega^\bullet(G)$ with a bigrading. In particular, each $\Omega^{p,q}(G)$ is an $\mathbb{R}$-module of bidegree $(p,q)\in\mathbb{N}\times\mathbb{Z}$. Moreover, in Lemma \ref{lem: Lemma 2.7} we established the existence of differential maps $d_i\colon\Omega^\bullet(G)\to\Omega^\bullet(G)$ of bidegree $\vert d_i\vert=(i,1-i)$ with $i=0,1,2,3$. We are left to show that the formulae \eqref{eq: structure maps multicomplex} also hold.

    These follow directly from the fact that $(\Omega^\bullet(G),d)$ is a complex and that forms of different weight are linearly independent by Lemma \ref{lem: red lemma}. Indeed, for any $k$-form $\alpha\in\Omega^{p,k-p}(G)$, we have the following explicit formula for $d^2\alpha$:
    \begin{align*}
        0=d^2\alpha=&(d_0+d_1+d_2+d_3)(d_0+d_1+d_2+d_3)\alpha\\=&\underbrace{d_0^2\alpha}_{\text{weight }p}+\underbrace{(d_0d_1+d_1d_0)\alpha}_{\text{weight }p+1}+\underbrace{(d_0d_2+d_1^2+d_2d_0)\alpha}_{\text{weight }p+2}+\underbrace{(d_0d_3+d_1d_2+d_2d_1+d_3d_0)\alpha}_{\text{weight }p+3}+\\&+\underbrace{(d_1d_3+d_2^2+d_3d_1)\alpha}_{\text{weight }p+4}+\underbrace{(d_2d_3+d_3d_2)\alpha}_{\text{weight }p+5}+\underbrace{d_3^2\alpha}_{\text{weight }p+6}\,.
    \end{align*}
    This means that for each $n=0,\ldots,6$, the form $\big(\sum_{i+j=n}d_id_j\big)\alpha$ belongs to the homogeneous component of weight $p+n$. Since the decomposition of $\Omega^{k+2}(G)$ according to the weight is direct by \eqref{eq: smooth forms according to weight and degree}, the uniqueness of this decomposition implies that $$\sum_{i+j=n}d_id_j\alpha=0\ \text{ for every }n=0,\ldots,6\,.$$  For $n>6$, the identity holds automatically because $d_i=0$ for $i>3$.
\end{proof}

Before moving on to study the properties of spectral complexes on the Engel group, let us stress the fact that in the proof of Proposition \ref{prop: truncated multicomplex} it is important that the ground ring of the multicomplex be taken to be $\mathbb{R}$. Indeed, in the original construction from \cite{livernet2020spectral}, the subspaces $Z_r^{p,k-p}$ and $B_r^{p,k-p}$ as in Definition \ref{def: Z and B defined} turn out to be $\mathbb K$-submodules of the original $\mathbb K$-module $\mathcal{C}$ as long as the maps $d_i\colon\mathcal{C}\to\mathcal{C}$ are $\mathbb K$-linear. In the case of a Carnot group such as our Engel group, only the algebraic differential $d_0$ turns out to be a $\mathcal{C}^\infty(G)$-linear map, while $d_1,d_2,d_3$ are differential maps which are only $\mathbb{R}$-linear. Therefore, for $r=1$ the spaces $Z_1^{p,k-p}$ and $B_1^{p,k-p}$ are $\C^\infty(G)$-submodules, since their definitions involve only the $\C^\infty(G)$-linear operator $d_0$. For $r>1$, one or more of the operators $d_1,d_2,d_3$ enter the construction, and so $Z_r^{p,k-p}$ and $B_r^{p,k-p}$ are in general only $\mathbb{R}$-vector subspaces and need not be $\C^\infty(G)$-submodules. 
\subsection{The Rumin complex} Even though the present paper only focuses on the behaviour of spectral complexes on the Engel group as introduced in \cite{tripaldi2026spectralcomplexestruncatedmulticomplexes}, their properties can be easily expressed in terms of the Rumin complex. For this reason, we will also briefly present the Rumin complex on the Engel group together with its main properties. A more thorough exposition of how to construct the Rumin complex in the Engel group can be found in \cite{rumin_grenoble,GagliardoNiremb,F+T1}.
\subsubsection{Introducing a scalar product on $\Omega^\bullet(G)$}\label{subsection: scalar product}
Once a scalar product has been fixed, the space of Rumin forms $E_0^\bullet$ can be realised as a graded subspace of $\Omega^\bullet(G)$ representing the cohomology of the complex $(\Omega^\bullet(G),d_0)$. Endowed with the differential $d_c$, it forms the Rumin complex $(E_0^\bullet,d_c)$. The identity $d_0^2=0$ follows directly from equation \eqref{eq: structure maps multicomplex} by taking $n=0$. Since we want to realise the $d_0$-cohomology as a subspace of smooth forms rather than as the quotient $\ker d_0/\operatorname{Im}d_0$, we need to select a complement of $\operatorname{Im}d_0$ inside $\ker d_0$. 

This can easily be obtained by introducing a scalar product on $\mathfrak{g}$ for which $X_1,X_2$ are orthonormal and extend it to $\mathfrak{g}$ by declaring the layers of the stratification \eqref{eq: stratification} mutually orthogonal so that the adapted basis $(X_1,X_2,X_3,X_4)$ is orthonormal. We orient $\mathfrak{g}$ according to this ordered basis, so that the dual basis $(\theta_1,\theta_2,\theta_3,\theta_4)$ is orthonormal and $\operatorname{vol}=\theta_1\wedge\theta_2\wedge\theta_3\wedge\theta_4$ is the associated unit volume form. This scalar product then canonically extends to the space of left-invariant forms $\Omega_L^\bullet(G)\cong\bigwedge^\bullet\mathfrak{g}^\ast$.

Since the scalar product $\langle\cdot,\cdot\rangle_k$ defined on each $\bigwedge^k\mathfrak{g}^\ast$ is adapted to the stratification, we have that covectors of different weight are orthogonal, i.e. for each degree $k=0,1,2,3,4$
\begin{align}\label{eq: diff weight implies orthog}
    \text{ given }\xi_1\in\bigwedge\nolimits^{p_1,k-p_1}\mathfrak{g}^\ast\ \text{ and }\ \xi_2\in\bigwedge\nolimits^{p_2,k-p_2}\mathfrak{g}^\ast\ ,\ \text{ if }p_1\neq p_2\text{ then }\langle\xi_1,\xi_2\rangle=0\,.
\end{align}
It is then possible to define the Hodge-$\star$ operator as the linear map acting on $\bigwedge^\bullet\mathfrak{g}^\ast$ as
\begin{align*}
    \star\colon\bigwedge\nolimits^k\mathfrak{g}^\ast\xrightarrow[]{\cong}\bigwedge\nolimits^{4-k}\mathfrak{g}^\ast\ ,\ \xi_1\wedge\star\xi_2:=\langle\xi_1,\xi_2\rangle_k\operatorname{vol}\ \text{ for any }\xi_1,\xi_2\in\bigwedge\nolimits^k\mathfrak{g}^\ast\,.
\end{align*}

The scalar product \(\langle\cdot,\cdot\rangle_k\) on \(\bigwedge^k\mathfrak g^\ast\) extends pointwise to a \(\mathcal C^\infty(G)\)-valued inner product on \(\Omega^k(G)\). More precisely, under the identification $
\Omega^k(G)
\cong
\mathcal C^\infty(G)\otimes\bigwedge\nolimits^k\mathfrak g^\ast$, we define
$$
\left\langle
\sum_i f_i\xi_i,\sum_j g_j\eta_j
\right\rangle_k
:=
\sum_{i,j}f_ig_j\langle\xi_i,\eta_j\rangle_k
\in\mathcal C^\infty(G).
$$
In particular, this extension is \(\mathcal C^\infty(G)\)-bilinear, since $
\langle f\alpha_1,g\alpha_2\rangle_k
=
fg\langle\alpha_1,\alpha_2\rangle_k
$ for every \(f,g\in\mathcal C^\infty(G)\) and
\(\alpha_1,\alpha_2\in\Omega^k(G)\).

The Hodge-\(\star\) operator also extends accordingly to a
\(\mathcal C^\infty(G)\)-linear map
$$
\star\colon\Omega^k(G)\xrightarrow[]{\cong}\Omega^{4-k}(G)\ ,\ 
\alpha_1\wedge\star\alpha_2
=
\langle\alpha_1,\alpha_2\rangle_k\operatorname{vol}\,.
$$
Since covectors of different weight are orthogonal by \eqref{eq: diff weight implies orthog}, smooth homogeneous forms of different weight are pointwise orthogonal:
$$
\text{given }\alpha_1\in\Omega^{p_1,k-p_1}(G)\ \text{ and }
\
\alpha_2\in\Omega^{p_2,k-p_2}(G)\ ,\ \text{ if }
p_1\neq p_2\ \text{ then }
\langle\alpha_1,\alpha_2\rangle_k=0\,.
$$

For a linear subspace
\(W\subseteq\bigwedge^k\mathfrak g^\ast\), let
$
\operatorname{pr}_W\colon
\bigwedge^k\mathfrak g^\ast
\to W
$
denote the corresponding finite-dimensional orthogonal projection. It extends \(\mathcal C^\infty(G)\)-linearly to the space of smooth forms:
$
\operatorname{Id}_{\mathcal C^\infty(G)}
\otimes\operatorname{pr}_W\colon
\Omega^k(G)
\to
\mathcal C^\infty(G)\otimes W
$.
Whenever
$
S=\mathcal C^\infty(G)\otimes W
\subseteq\Omega^k(G)
$
is a submodule of this form, we denote this fibrewise orthogonal projection by \(\operatorname{pr}_S\). The orthogonal complements and projections associated below with
\(\operatorname{Im}d_0\), \(\operatorname{Im}\delta_0\), and \(E_0^\bullet\) are understood in this fibrewise sense (see Definitions \ref{def: delta_0}, \ref{def: partial inverse of d_0}, and \ref{def: E_0}).

Separately, the pointwise inner product induces a scalar-valued \(L^2\)-inner product on compactly supported smooth forms. For
\(\alpha_1,\alpha_2\in\Omega_c^k(G)\), define
$$
(\alpha_1,\alpha_2)_{L^2}
:=
\int_G
\langle\alpha_1,\alpha_2\rangle_k\operatorname{vol}
=
\int_G\alpha_1\wedge\star\alpha_2\,.
$$

The integral is finite because both forms are compactly supported. The completion of $\Omega^k_c(G)$ with respect to the induced norm gives the Hilbert space \(L^2\Omega^k(G)\). The same integral is also well defined as a pairing whenever one of the two smooth forms is compactly supported.

The fibrewise projections introduced above preserve smoothness and compact support and extend to bounded orthogonal projections on \(L^2\Omega^k(G)\). We will keep the notation
\(\langle\cdot,\cdot\rangle_k\) to refer to the pointwise \(\mathcal C^\infty(G)\)-valued inner product, whereas
\((\cdot,\cdot)_{L^2}\) will denote the integrated \(L^2\)-inner product.

It is also useful to point out the effect of the Hodge-$\star$ operator on the weight decomposition. Since the Hodge-$\star$ maps forms of degree $k$ to forms of degree $4-k$ and changes their weight from $p$ to $Q-p=7-p$, we have
\begin{align}\label{eq: Hodge star and weights}
\star\bigl(\Omega^{p,k-p}(G)\bigr)
=
\Omega^{7-p,p-k-3}(G).
\end{align}

\begin{definition}[The adjoint of $d_0$]\label{def: delta_0} We can use the scalar product $\langle\cdot,\cdot\rangle_k$ just introduced to define the fibrewise or algebraic adjoint of $d_0$, which we will denote by $\delta_0$. In other words, the map $\delta_0\colon\Omega^{p,k-p}(G)\to\Omega^{p,k-1-p}(G)$ is defined by imposing
\begin{align*}
    \langle d_0\alpha_1,\alpha_2\rangle_{k}=\langle\alpha_1,\delta_0\alpha_2\rangle_{k-1}\ \text{ for any }\alpha_1\in\Omega^{k-1}(G)\text{ and }\alpha_2\in\Omega^{k}(G)\,.
\end{align*}    
\end{definition}
The fact that $\delta_0$ keeps the weight of forms constant, i.e. $\delta_0\big(\Omega^{p,q}(G)\big)\subset\Omega^{p,q-1}(G)$, is a direct consequence of the fact that elements of different weight are orthogonal and that $d_0$ keeps the weight of forms constant, since $\vert d_0\vert=(0,1)$.

Another crucial operator that can be defined from $d_0$ using the scalar product is its partial inverse $d_0^{-1}$. This operator plays a central role in the construction of the Rumin differential $d_c$.
\begin{definition}[The partial inverse $d_0^{-1}$]\label{def: partial inverse of d_0}
For every weight $p$ and degree $k$, the finite-dimensional orthogonal decomposition on $\bigwedge^\bullet\mathfrak{g}^\ast$, extended $\mathcal{C}^\infty(G)$-linearly to smooth forms, implies that the restriction
\begin{align*}
d_0\big\vert_{\operatorname{Im}\delta_0\cap\Omega^{p,k-1-p}(G)}
\colon
\operatorname{Im}\delta_0\cap\Omega^{p,k-1-p}(G)
\xrightarrow{\ \cong\ }
\operatorname{Im}d_0\cap\Omega^{p,k-p}(G)
\end{align*}
is an isomorphism of $\mathcal{C}^\infty(G)$-modules. Indeed, $
\operatorname{Im}\delta_0=(\ker d_0)^\perp$,
so the restriction is injective.

Following Rumin's construction, it is customary to use the shorthand notation to denote the partial inverse of $d_0$ as the $\C^\infty(G)$-linear map
\begin{align*}
d_0^{-1}\colon
\Omega^{p,k-p}(G)
&\longrightarrow
\Omega^{p,k-1-p}(G)\ ,\
d_0^{-1}
:=
\left(
d_0\big\vert_{\operatorname{Im}\delta_0}
\right)^{-1}
\circ\operatorname{pr}\big\vert_{\operatorname{Im}d_0}\,.
\end{align*}
Thus, $d_0^{-1}$ coincides with the inverse of $d_0$ on $\operatorname{Im}d_0$ and vanishes on $(\operatorname{Im}d_0)^\perp$. 
In particular, it readily follows that $(d_0^{-1})^2=0$ and $\ker d_0^{-1}=\ker\delta_0=(\operatorname{Im}d_0)^\perp$.
\end{definition}

Before defining the space of Rumin forms, let us stress that the map $d_0\colon\Omega^k(G)\to\Omega^{k+1}(G)$ is the $\mathcal C^\infty(G)$-linear extension of the Chevalley-Eilenberg differential $d_\mathfrak{g}$ on left-invariant forms. More precisely, under the identification
$
\Omega^\bullet(G)
\cong
\mathcal C^\infty(G)\otimes\bigwedge\nolimits^\bullet\mathfrak g^\ast$,we have
$$
d_0
=
\operatorname{Id}_{\mathcal C^\infty(G)}
\otimes d_{\mathfrak{g}},
$$
where $d_{\mathfrak{g}}$ acts between finite-dimensional spaces. Therefore, $
\operatorname{Im}d_0
=
\mathcal C^\infty(G)\otimes\operatorname{Im}d_{\mathfrak{g}}$,
so that $\operatorname{Im}d_0$ is a fibrewise complemented subspace of $\Omega^\bullet(G)$. Moreover, $d_0$ extends to a bounded operator
$$
d_0\colon
L^2\Omega^{p,k-p}(G)
\to
L^2\Omega^{p,k+1-p}(G)\,,
$$
whose image is
$
L^2(G)\otimes
\operatorname{Im}d_\mathfrak{g}$.
Since $\operatorname{Im}d_{\mathfrak{g}}$ is finite-dimensional, this image is closed in the $L^2$ topology.

\begin{definition}[The space of Rumin forms]\label{def: E_0} As shown in Proposition \ref{prop: truncated multicomplex}, the de Rham complex on the Engel group $(\Omega^\bullet(G),d)$ is a truncated multicomplex with $d=d_0+d_1+d_2+d_3$. In particular, equation \eqref{eq: structure maps multicomplex} for $n=0$ gives that $d_0^2=0$, that is $(\Omega^\bullet(G),d_0)$ is also a complex. Therefore, it is possible to compute its cohomology. As already mentioned, we are interested in identifying subspaces of forms instead of quotients $\ker d_0/\operatorname{Im}d_0$. Using the fibrewise scalar product just introduced, one can define the space of Rumin forms $E_0^\bullet\subset\Omega^\bullet(G)$ as
\begin{align}
    E_0^k=\ker d_0\cap (\operatorname{Im}d_0)^\perp\cap\Omega^k(G)=\ker d_0\cap\ker\delta_0\cap\Omega^k(G)\,,
\end{align}
the last equality following directly from the defining adjoint relation for $\delta_0$, which gives $(\operatorname{Im}d_0)^\perp=\ker\delta_0$.
\end{definition}
As a direct consequence of Definition \ref{def: E_0} we then get that the map
\[
E_0^k\longrightarrow H^k\bigl(\Omega^\bullet(G),d_0\bigr)\ ,
\
\alpha\longmapsto[\alpha]
\]
is an isomorphism. 

\begin{definition}[Orthogonal projection onto $E_0^\bullet$]\label{def: Pi_0}
We define
\begin{align*}
\Pi_0
&:=
\operatorname{Id}
-d_0d_0^{-1}
-d_0^{-1}d_0
\colon
\Omega^\bullet(G)
\longrightarrow
\Omega^\bullet(G).
\end{align*}
By Definition \ref{def: partial inverse of d_0}, the operators
$
d_0d_0^{-1}
=
\operatorname{pr}_{\operatorname{Im}d_0}$ and $
d_0^{-1}d_0
=
\operatorname{pr}_{\operatorname{Im}\delta_0}
$
are the fibrewise orthogonal projections onto $\operatorname{Im}d_0$ and $\operatorname{Im}\delta_0$, respectively. Moreover, $\Pi_0^2=(\operatorname{Id}-d_0d_0^{-1}-d_0^{-1}d_0)(\operatorname{Id}-d_0d_0^{-1}-d_0^{-1}d_0)=\operatorname{Id}-d_0d_0^{-1}-d_0^{-1}d_0=\Pi_0$, hence $\Pi_0$ is the fibrewise orthogonal projection onto the orthogonal complement of
$
\operatorname{Im}d_0\oplus\operatorname{Im}\delta_0$.

Finally, since $(\operatorname{Im}d_0)^\perp=\ker\delta_0$ and $(\operatorname{Im}\delta_0)^\perp=\ker d_0$,
we obtain $\operatorname{Im}\Pi_0
=
\ker d_0\cap\ker\delta_0
=
E_0^\bullet.
$, that is $\Pi_0=\operatorname{pr}_{E_0^\bullet}$. Moreover, as a direct consequence of the definition of this projection, we get that $\Pi_0$ preserves both weight and degree, that is
$\Pi_0\bigl(\Omega^{p,k-p}(G)\bigr)
\subseteq
E_0^k\cap\Omega^{p,k-p}(G)$.
\end{definition}

A useful alternative characterisation of the space of Rumin forms is obtained by introducing the algebraic Laplacian associated with $d_0$.

\begin{definition}[The algebraic Laplacian $\Box_0$]\label{def: Box_0}
The \emph{algebraic Laplacian} associated with $d_0$ is the fibrewise operator
\begin{align*}
\Box_0
&:=
d_0\delta_0+\delta_0d_0
\colon
\Omega^\bullet(G)
\longrightarrow
\Omega^\bullet(G)\,.
\end{align*}
Both summands have bidegree $(0,0)$, hence, just like the projection $\Pi_0$, the map $\Box_0$ preserves both the weight and degree of forms.
\end{definition}

The operator $\Box_0$ is fibrewise self-adjoint and non-negative. Indeed, for every $\alpha,\beta\in\Omega^k(G)$,
\begin{align*}
\langle\Box_0\alpha,\beta\rangle_k
&=
\langle\delta_0\alpha,\delta_0\beta\rangle_{k-1}
+
\langle d_0\alpha,d_0\beta\rangle_{k+1}
=
\langle\alpha,\Box_0\beta\rangle_k.
\end{align*}
Moreover, since
$\langle\Box_0\alpha,\alpha\rangle_k
=
\langle d_0\alpha,d_0\alpha\rangle_{k+1}
+
\langle\delta_0\alpha,\delta_0\alpha\rangle_{k-1}$,
it follows that
\begin{align}\label{eq: Rumin forms as harmonic rep}
\ker\Box_0
=
\ker d_0\cap\ker\delta_0
=
E_0^\bullet.
\end{align}

\begin{proposition}[Hodge decomposition for $\Box_0$]
For every degree $k$, the space of smooth $k$-forms admits the fibrewise orthogonal decomposition
\begin{align}\label{eq: Hodge decomp Box_0 degree}
\Omega^k(G)
=
\operatorname{Im}\bigl(
d_0\colon\Omega^{k-1}(G)\to\Omega^k(G)
\bigr)
\oplus
E_0^k
\oplus
\operatorname{Im}\bigl(
\delta_0\colon\Omega^{k+1}(G)\to\Omega^k(G)
\bigr).
\end{align}
Equivalently,
\begin{align}\label{eq: Hodge decomp Box_0}
\Omega^\bullet(G)
=
\operatorname{Im}d_0
\oplus
\ker\Box_0
\oplus
\operatorname{Im}\delta_0.
\end{align}
\end{proposition}
\begin{proof}
All the orthogonal complements in this proof are understood fibrewise, and the kernels and images are taken in the appropriate degrees. Since $d_0$ acts fibrewise as a linear map between finite-dimensional spaces, we have
$(\ker d_0)^\perp=\operatorname{Im}\delta_0$.
Consequently, for every degree $k$,
        \begin{align*}
            \Omega^k(G)=&\ker d_0\oplus(\ker d_0)^\perp=\ker d_0\oplus\operatorname{Im}\delta_0\\=&\ker d_0\cap\ker \delta_0\oplus\ker d_0\cap(\ker \delta_0)^\perp\oplus\operatorname{Im}\delta_0\\=&\ker\Box_0\oplus\ker d_0\cap\operatorname{Im}d_0\oplus\operatorname{Im}\delta_0=\ker\Box_0\oplus\operatorname{Im}d_0\oplus\operatorname{Im}\delta_0\,.
        \end{align*}
    \end{proof}

In order to keep the exposition as short as possible, we will directly express the Rumin differential $d_c$ directly in terms of the structure maps of the truncated multicomplex. The standard construction of $d_c$ is due to Rumin \cite{rumin_grenoble}, while the equivalence between that construction and the expression in Proposition \ref{prop: d_c on forms} below is proved in Lemma 2.19 of \cite{tripaldi2026spectralcomplexestruncatedmulticomplexes} or, equivalently, in Lemma 3.7 of \cite{magnani2026stokestheorempositivelygraded}.

\begin{definition}[The operators $\partial_r$]\label{def: operators curly}
Let $d_0,d_1,d_2,d_3$ be the components of the exterior differential introduced in Lemma \ref{lem: Lemma 2.7}. Since the de Rham complex $(\Omega^\bullet(G),d)$ is a truncated 3-multicomplex (see Proposition \ref{prop: truncated multicomplex}), let us set $d_r=0$ for every $r>0$. Using the partial inverse $d_0^{-1}$ of Definition \ref{def: partial inverse of d_0}, we define recursively the operators $\partial_r\colon\Omega^k(G)\to\Omega^{k+1}(G)$
by
  \begin{align}\label{eq: operators curly}
        \partial_1=d_1\ \text{ and }\ \partial_r=d_r-\sum_{j=1}^{r-1}d_{r-j}d_0^{-1}\partial_j\ \text{ for }r\ge 2\,.
    \end{align}
\end{definition}

Since $d_0^{-1}$ has bidegree $(0,-1)$, one can easily show that $\partial_r$ has bidegree $\lvert\partial_r\rvert=(r,1-r)$, that is\begin{align}\label{eq: bidegree partial r}
        \text{ for any }\alpha\in\Omega^{p,k-p}(G)\ \text{ we have that }\ \partial_r\alpha\in\Omega^{p+r,k+1-p-r}(G)\,.
    \end{align}
Thus, $\partial_r$ increases the weight by $r$ and the total degree by $1$.

From the explicit weight decomposition of forms on the Engel group, the maximum possible increase in weight between two consecutive form degrees is $4$, and so by \eqref{eq: bidegree partial r} we have that in this specific group
\[
\partial_r=0\ \text{ for every }r\geq 5\,.
\]
\begin{proposition}[Expressing $d_c$ on $E_0^\bullet$]\label{prop: d_c on forms}
For every $\alpha\in E_0^\bullet$, the Rumin differential is given by
\begin{align}\label{eq: formulation of d_c}
d_c\alpha
&=
\Pi_0\sum_{r=1}^{4}\partial_r\alpha
=
\sum_{r=1}^{4}\partial_r\alpha
-d_0d_0^{-1}\sum_{r=1}^{4}\partial_r\alpha
-d_0^{-1}d_0\sum_{r=1}^{4}\partial_r\alpha\,.
\end{align}
\end{proposition}

\begin{definition}[The homogeneous components of $d_c$]\label{def: d_c^r}
For each $r=1,\ldots,4$, we define
\begin{align*}
d_c^r=\partial_r-d_0^{-1}d_0\partial_r-d_0d_0^{-1}\partial_r\,.
\end{align*}
The operator $d_c^r$ has bidegree $(r,1-r)$ and therefore restricts to a map
\[
d_c^r\colon
E_0^k\cap\Omega^{p,k-p}(G)
\longrightarrow
E_0^{k+1}\cap\Omega^{p+r,k+1-p-r}(G).
\]
Moreover, with this notation,
\[
d_c=\sum_{r=1}^{4}d_c^r.
\]
\end{definition}

\subsubsection{The Rumin complex in the Engel group}\label{subsection: Rumin complex in Engel}Let us focus our attention on the case of the Engel group. Given the explicit action of the algebraic differential $d_0$ as the differential $d$ acting on left-invariant forms
\begin{align*}
    d\theta_1=d\theta_2=0\ ,\ d\theta_3=-\theta_1\wedge\theta_2\ ,\ d\theta_4=-\theta_1\wedge\theta_3\,,
\end{align*}
one readily obtains an explicit description of the space of smooth forms through the direct sum decomposition \eqref{eq: Hodge decomp Box_0} in each degree.
\begin{itemize}
    \item $\Omega^0(G)=\mathcal{C}^\infty(G)=\ker\Box_0\cap\Omega^0(G)$;
    \item $\Omega^1(G)=(\ker\Box_0\oplus\operatorname{Im}\delta_0)\cap\Omega^1(G)$, where $\ker\Box_0\cap\Omega^1(G)=\Omega^{1,0}(G)=\operatorname{span}_{\mathcal{C}^\infty(G)}\{\theta_1,\theta_2\}$, $\operatorname{Im}\delta_0\cap\Omega^1(G)=\operatorname{span}_{\mathcal{C}^\infty(G)}\{\theta_3,\theta_4\}$, and $\operatorname{Im}d_0\cap\Omega^1(G)=0$;
    \item $\Omega^2(G)=(\operatorname{Im}d_0\oplus\ker\Box_0\oplus\operatorname{Im}\delta_0)\cap\Omega^2(G)$, where $\ker\Box_0\cap\Omega^2(G)=\operatorname{span}_{\mathcal{C}^\infty(G)}\{\theta_2\wedge\theta_3,\theta_1\wedge\theta_4\}$, $\operatorname{Im}d_0\cap\Omega^2(G)=\operatorname{span}_{\mathcal{C}^\infty(G)}\{\theta_1\wedge\theta_2,\theta_1\wedge\theta_3\}$, and $\operatorname{Im}\delta_0\cap\Omega^2(G)=\operatorname{span}_{\mathcal{C}^\infty(G)}\{\theta_2\wedge\theta_4,\theta_3\wedge\theta_4\}$;
    \item $\Omega^3(G)=(\ker\Box_0\oplus\operatorname{Im}d_0)\cap\Omega^3(G)$, where $\ker\Box_0\cap\Omega^3(G)=\operatorname{span}_{\mathcal{C}^\infty(G)}\{\theta_1\wedge\theta_3\wedge\theta_4,\theta_2\wedge\theta_3\wedge\theta_4\}$, $\operatorname{Im}d_0\cap\Omega^3(G)=\operatorname{span}_{\mathcal{C}^\infty(G)}\{\theta_1\wedge\theta_2\wedge\theta_3,\theta_1\wedge\theta_2\wedge\theta_4\}$, and $\operatorname{Im}\delta_0\cap\Omega^3(G)=0$;
    \item $\Omega^4(G)=\ker\Box_0\cap\Omega^4(G)=\Omega^{7,-3}(G)=\operatorname{span}_{\mathcal{C}^\infty(G)}\{\theta_1\wedge\theta_2\wedge\theta_3\wedge\theta_4\}$.
\end{itemize}
Given this characterisation of the space of smooth forms, together with \eqref{eq: Rumin forms as harmonic rep}, we have that
\begin{itemize}
    \item $E_0^0=\mathcal{C}^\infty( G)=\Omega^{0,0}(G)$;
    \item $E_0^1=\mathrm{span}_{\mathcal C^\infty( G)}\lbrace\theta_1,\theta_2\rbrace=\Omega^{1,0}(G)$;
    \item $E_0^2=\mathrm{span}_{\mathcal C^\infty( G)}\lbrace \theta_2\wedge\theta_3,\theta_1\wedge\theta_4\rbrace\subset\Omega^{3,-1}(G)\oplus\Omega^{4,-2}(G)$;
    \item $E_0^3=\mathrm{span}_{\mathcal C^\infty(G)}\lbrace \theta_1\wedge\theta_3\wedge\theta_4,\theta_2\wedge\theta_3\wedge\theta_4\rbrace=\Omega^{6,-3}(G)$;
    \item $E_0^4=\mathrm{span}_{\mathcal C^\infty( G)}\lbrace \theta_1\wedge\theta_2\wedge\theta_3\wedge\theta_4\rbrace=\Omega^{7,-3}(G)$.
\end{itemize}
By studying simply the difference in weight and relating it to the bidegree of the operators $d_c^r$, we get that
\begin{itemize}
    \item $d_c\colon E_0^0\longrightarrow E_0^1$ is given by $d_c=d_c^1$;
    \item $d_c\colon E_0^1\longrightarrow E_0^2$ is given by $d_c=d_c^2+d_c^3$;
    \item $d_c\colon E_0^2\longrightarrow E_0^3$ is given by $d_c=d_c^2+d_c^3$;
    \item $d_c\colon E_0^3\longrightarrow E_0^4$ is given by $d_c=d_c^1$.
\end{itemize}
Furthermore, once we apply the expressions \eqref{eq: operators curly} to the formulae in Definition \ref{def: d_c^r}, it is a straightforward computation to find the action of the Rumin differentials.
\begin{itemize}
    \item $k=0$: given an arbitrary 0-form $f\in E_0^0$, we have 
    \begin{align*}
        d_cf=d_c^1f=d_1f-d_0d_0^{-1}d_1f-d_0^{-1}d_0d_1f=d_1f=X_1f\theta_1+X_2f\theta_2\,.
    \end{align*}
    \item $k=1$: given an arbitrary Rumin 1-form $\alpha=f_1\theta_1+f_2\theta_2\in E_0^1$, we have
    \begin{align*}
        d_c\alpha=&d_c^2\alpha+d_c^3\alpha=(\partial_2-d_0d_0^{-1}\partial_2-d_0^{-1}d_0\partial_2)\alpha+(\partial_3-d_0d_0^{-1}\partial_3-d_0^{-1}d_0\partial_3)\alpha\\=&
\left[
X_2\bigl(X_1f_2-X_2f_1\bigr)-X_3f_2
\right]\theta_2\wedge\theta_3+
\left[
X_1^2\bigl(X_1f_2-X_2f_1\bigr)-X_1X_3f_1
-X_4f_1
\right]\theta_1\wedge\theta_4.
    \end{align*}
    \item $k=2$: given an arbitrary Rumin 2-form $\alpha=f_1\theta_2\wedge\theta_3+f_2\theta_1\wedge\theta_4\in E_0^2$, we have
    \begin{align*}
        d_c\alpha=&d_c^2(f_2\theta_1\wedge\theta_4)+d_c^3(f_1\theta_2\wedge\theta_3)=-(X_1X_2f_2+X_3f_2)\theta_1\wedge\theta_3\wedge\theta_4-X_2^2f_2\theta_2\wedge\theta_3\wedge\theta_4\\&+X_1^3f_1\theta_1\wedge\theta_3\wedge\theta_4+(X_2X_1^2f_1-X_3X_1f_1+X_4f_1)\theta_2\wedge\theta_3\wedge\theta_4\,.
    \end{align*}
    \item $k=3$: given an arbitrary Rumin 3-form $\alpha=f_1\theta_1\wedge\theta_3\wedge\theta_4+f_2\theta_2\wedge\theta_3\wedge\theta_4\in E_0^3$, we have
    \begin{align*}
        d_c\alpha=&d_c^1\alpha=\left(X_1f_2-X_2f_1\right)
\theta_1\wedge\theta_2\wedge\theta_3\wedge\theta_4\,.
    \end{align*}
\end{itemize}
Adopting the pictorial convention commonly used for differential complexes on Carnot groups, the Rumin complex on the Engel group can be represented by the following diagram. The horizontal coordinate records the degree of a form, which increases from left to right, while the vertical coordinate records its weight, which increases from top to bottom. Thus, the diagram begins in the upper-left corner with the space $E_0^0=\Omega^{0,0}(G)$ of $0$-forms of weight $0$. Using this convention, the diagram represents the Rumin complex on the Engel group, where only the non-trivial homogeneous components $E_0^\bullet\cap\Omega^{p,k-p}(G)$ are included, and the arrows indicate the operators $d_c^r$ which are non-zero when acting on the respective spaces of Rumin forms.
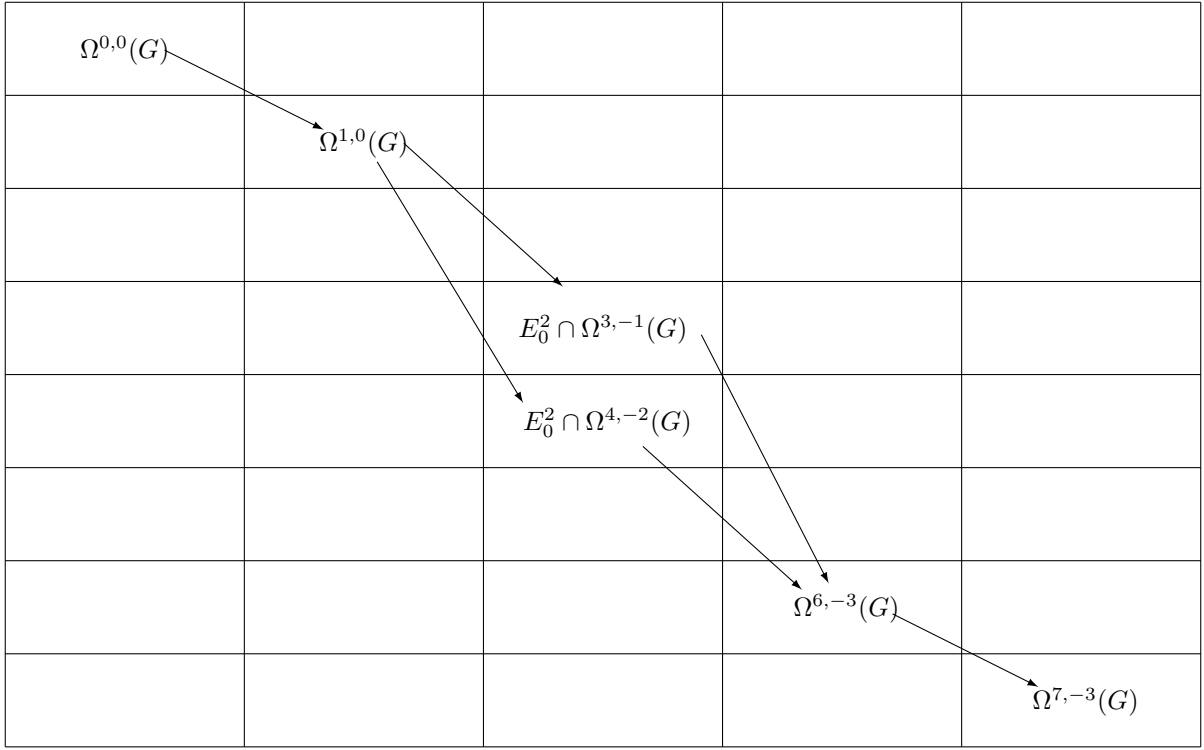
\begin{align*}
    \begin{picture}(450,280)(0,0)
   \put(  0,  0){\line( 1, 0){450}}
   \put(  0,  0){\line( 0, 1){280}}
   \multiput(  0, 35)( 0,35){8}{\line( 1, 0){450}}
   \multiput( 90,  0)(90, 0){5}{\line( 0, 1){280}}
   \put(45,262){\makebox(0,0){$ \Omega^{0,0}(G)$}}
   \put(135,227){\makebox(0,0){$ \Omega^{1,0}(G)$}}
   \put(225,157){\makebox(0,0){$E_0^2\cap\Omega^{3,-1}(G)$}}
   \put(225,122){\makebox(0,0){{ $E_0^{2}\cap\Omega^{4,-2}(G)$}}}
   \put(315,52){\makebox(0,0){{ $\Omega^{6,-3}(G)$}}}
   \put(405,17){\makebox(0,0){{ $\Omega^{7,-3}(G)$}}}
   \put(60,262){\vector(1,-0.5){60}}
   \put(150,227){\vector(1,-0.9){60}}
   \put(140,220){\vector(1,-1.65){55}}
   \put(262,155){\vector(1,-1.95){48}}
   \put(240,113){\vector(1,-0.9){60}}
   \put(334,50){\vector(1,-0.5){55}}
  \end{picture} 
\end{align*}

\subsection{Spectral complexes on the Engel group}

In this subsection, we focus on the spectral complexes associated with the truncated $3$-multicomplex $
\bigl(\Omega^\bullet(G),d_0,d_1,d_2,d_3\bigr)$ arising from the de Rham complex of the Engel group. We begin by recalling the $\mathbb R$-submodules $Z_r^{p,k-p}$ and $B_r^{p,k-p}$ whose quotients represent the terms of the associated spectral sequence and constitute the building blocks of the spectral complexes. In view of the applications developed in the subsequent sections, we present both their original description in terms of the structure maps $d_i$ and their alternative characterisation in terms of Rumin forms and the homogeneous components $d_c^r$ of the Rumin differential given in \cite{tripaldi2026spectralcomplexestruncatedmulticomplexes}.

\begin{definition}[Adapted from Definition 2.6 in \cite{livernet2020spectral}]\label{def: Z and B defined}
Let $\alpha\in\Omega^{p,k-p}(G)$ and let $r\ge 1$. We define the bigraded $\mathbb{R}$-submodules $Z_r^{p,k-p}$ and $B_r^{p,k-p}$ of $\Omega^{p,k-p}(G)$ as follows:
\begin{align*}
    \alpha\in Z_r^{p,k-p}\ \Longleftrightarrow&\ \text{for }1\le j\le r-1\,,\text{ there exists }z_{p+j}\in\Omega^{p+j,k-p-j}(G)\text{ such that}\\&\ d_0\alpha=0\text{ and }d_n\alpha=\sum_{i=0}^{n-1}d_iz_{p+n-i}\text{ for all }1\le n\le r-1\,.\\
    \alpha\in B_r^{p,k-p}\ \Longleftrightarrow&\ \text{for }0\le j\le r-1\text{ there exists }c_{p-j}\in\Omega^{p-j,k-1-p+j}(G)\text{ such that}\\&\ \alpha=\sum_{j=0}^{r-1}d_jc_{p-j}\text{ and }0=\sum_{j=l}^{r-1}d_{j-l}c_{p-j}\text{ for }1\le l\le r-1\,.
\end{align*}
For $r=1$, the families indexed by $j=1,\ldots,r-1$ and the corresponding compatibility conditions are understood to be empty.
\end{definition}

\begin{remark}\label{rmk: Z_1 and B_1}
For $r=1$, Definition \ref{def: Z and B defined} reduces to
\begin{align*}
Z_1^{p,k-p}
&=
\ker d_0\cap\Omega^{p,k-p}(G)\ \text{ and }\
B_1^{p,k-p}
=
\operatorname{Im}d_0\cap\Omega^{p,k-p}(G)\,.
\end{align*}
More explicitly, if $\alpha\in\Omega^{p,k-p}(G)$, then
\begin{align*}
\alpha\in Z_1^{p,k-p}
 &\ \Longleftrightarrow\ 
d_0\alpha=0\
\Longleftrightarrow
\ \alpha=d_0\beta_p+\Pi_0\alpha
\ \text{ for some }
\beta_p\in\Omega^{p,k-1-p}(G)\,.
\end{align*}Similarly,
$$
\alpha\in B_1^{p,k-p}
\Longleftrightarrow
\alpha=d_0c_p
\ \text{ for some }
c_p\in\Omega^{p,k-1-p}(G)\,.
$$

The fibrewise Hodge decomposition therefore gives the orthogonal direct sum
$$
Z_1^{p,k-p}
=
B_1^{p,k-p}
\oplus
\bigl(E_0^k\cap\Omega^{p,k-p}(G)\bigr)\,,
$$
and hence a natural identification
$
Z_1^{p,k-p}/B_1^{p,k-p}
\cong
E_0^k\cap\Omega^{p,k-p}(G)
$.

\end{remark}

Using the Hodge decomposition associated with the algebraic Laplacian $
\Box_0=d_0\delta_0+\delta_0d_0,
$
the submodules $Z_r^{p,k-p}$ and $B_r^{p,k-p}$ can also be characterised in terms of Rumin forms and the homogeneous components $d_c^r$ of the Rumin differential. These characterisations are given by Propositions 3.4 and 3.6 in \cite{tripaldi2026spectralcomplexestruncatedmulticomplexes} and will be recalled below.

\begin{proposition}[Characterising $Z_r^{p,k-p}$ and $B_r^{p,k-p}$ via the Rumin complex]\label{prop: Z and B in terms of d_c} Let $r\ge 2$ and let $\alpha\in\Omega^{p,k-p}(G)$. The condition $\alpha\in Z_r^{p,k-p}$ is equivalent to saying that $\alpha=d_0\beta_p+\Pi_0\alpha$ for some $\beta_p\in\Omega^{p,k-1-p}(G)$ and there exist $\omega_{p+i}\in\Omega^{p+i,k-p-i}(G)\cap\ker\Box_0$ with $i=1,\ldots,r-2$ with \begin{align*}
    d_c^i\Pi_0\alpha=\sum_{j=1}^{i-1}d_c^{i-j}\omega_{p+j}\ \text{ for each }i=1,\ldots,r-1\,.
\end{align*}
On the other hand, the condition $\alpha\in B_r^{p,k-p}$ is equivalent to saying that there exists $c_{p-r+1}\in Z_{r-1}^{p-(r-1),k-2-p+r}$, that is $c_{p-r+1}=d_0\beta_{p-r+1}+\Pi_0c_{p-r+1}$ for some $\beta_{p-r+1}\in\Omega^{p-r+1,k-3-p+r}(G)$ and $\omega_{p-r+i}\in\Omega^{p-r+i,k-1-p+r-i}(G)\cap\ker\Box_0$ with $i=2,\ldots,r-1$ such that
\begin{align*}
    d_c^i\Pi_0c_{p-r+1}=\sum_{j=1}^{i-1}d_c^{i-j}\omega_{p-r+1+j}\ \text{ for each }i=1,\ldots,r-2,
\end{align*}
such that
\begin{align*}
    \alpha\equiv
d_c^{r-1}\Pi_0c_{p-r+1}
-\sum_{i=1}^{r-2}d_c^{r-1-i}\omega_{p-r+1+i}
\pmod{\operatorname{Im}d_0\cap\Omega^{p,k-p}(G)}\,.
\end{align*}
    We recall that in the case of the Engel group, since $\partial_r=0$ for $r\ge 5$, we have $d_c^r=0$ for the same range.
\end{proposition}

Using the graded submodules $Z_r^{p,k-p}$ and $B_r^{p,k-p}$, it is possible to have an explicit formulation of the differentials arising at each page of the spectral sequence. Again, using Proposition \ref{prop: Z and B in terms of d_c}, it is possible to express them using the Rumin differentials $d_c^r$. We refer to Theorem 2.10 in \cite{livernet2020spectral} and Propositions 3.7 and 3.8 in \cite{tripaldi2026spectralcomplexestruncatedmulticomplexes} for the proofs of the two formulas in the following proposition.
\begin{proposition}\label{prop: differentials Delta_r}Let $r\ge 1$. The $r^{th}$ differential of the spectral sequence corresponds to the map
\begin{align*}
    \Delta_r\colon Z_r^{p,k-p}/B_r^{p,k-p}\longrightarrow Z_r^{p+r,k+1-p-r}/B_r^{p+r,k+1-p-r}\ , \ \Delta_r\left([\alpha]\right)=\left[d_r\alpha-\sum_{i=1}^{r-1}d_iz_{p+r-i}\right]
    \end{align*}
    where $\alpha\in Z_r^{p,k-p}$ and the family $\lbrace z_{p+j}\rbrace_{1\le j\le r-1}$ satisfies the equations of Definition \ref{def: Z and B defined}.

    The same operator can be expressed as
    \begin{align}
        \Delta_r\left([\alpha]\right)=\left[d_c^r\Pi_0{\alpha}-\sum_{i=2}^{r-1}d_c^{i}{\omega}_{p+r-i}\right]\,,
    \end{align}
    where ${\omega}_{p+r-i}\in\ker\Box_0\cap\Omega^{p+r-i,k-p-r+i}(G)$ are taken to satisfy $d_c^j\Pi_0{\alpha}-\sum_{i=1}^{j-1}d_c^i{\omega}_{p+j-i}=0$ for $j=1,\ldots,r-1$, as presented in Proposition \ref{prop: Z and B in terms of d_c}.
    
\end{proposition}

As observed in Remark \ref{rmk: Z_1 and B_1}, the first page of the spectral sequence is naturally isomorphic to the space of Rumin forms. More precisely, the fibrewise scalar product and the Hodge decomposition \eqref{eq: Hodge decomp Box_0} give
\begin{align*}
E_1^{p,k-p}=Z_1^{p,k-p}/B_1^{p,k-p}
\cong
Z_1^{p,k-p}\cap
\left(B_1^{p,k-p}\right)^\perp=
E_0^k\cap\Omega^{p,k-p}(G).
\end{align*}

For each bidegree $(p,k-p)$, define
\begin{align}\label{eq: admissible orders}
I_{(p,k)}
:=
\left\{
j\geq 1\mid
\left.d_c^j\right|_{E_0^k\cap\Omega^{p,k-p}(G)}
\not\equiv 0
\right\}\,.
\end{align}
Then
\begin{align*}
\left.d_c\right|_{E_0^k\cap\Omega^{p,k-p}(G)}
=
\sum_{j\in I_{(p,k)}}d_c^j\colon E_0^{k}\cap\Omega^{p,k-p}(G)\longrightarrow\bigoplus_{j\in I_{(p,k)}}E_0^{k+1}\cap\Omega^{p+j,k+1-p-j}(G)\,.
\end{align*}
In the case of the Engel group, $I_{(p,k)}\subseteq\{1,2,3\}$.

Before introducing the spectral complexes, we recall the following inclusions.

\begin{lemma}[Lemma 4.1 in \cite{tripaldi2026spectralcomplexestruncatedmulticomplexes}]
\label{lem: inclusions B and Z}
For every bidegree \((p,q)\), the submodules \(B_r^{p,q}\) form an increasing family, whereas the submodules \(Z_r^{p,q}\) form a decreasing family. Moreover,
\begin{align}\label{eq: B contained in Z}
B_l^{p,q}\subseteq Z_j^{p,q}
\
\text{ for every }j,l\ge 1\,.
\end{align}
In particular, the quotient \(Z_j^{p,q}/B_l^{p,q}\) is well defined for arbitrary positive integers \(j\) and \(l\).
\end{lemma}

The following result is the crucial observation underlying the construction of the spectral complexes.

\begin{proposition}[Proposition 4.2 in
\cite{tripaldi2026spectralcomplexestruncatedmulticomplexes}]
\label{prop: mixed page differentials}
Let \(j\in I_{(p,k)}\). For every pair of positive integers \(m_1,m_2\), the spectral-sequence differential described in Proposition
\ref{prop: differentials Delta_r} induces a well-defined map
\begin{align}\label{eq: mixed page differential}
\Delta_j\colon
\frac{Z_j^{p,k-p}}{B_{m_1}^{p,k-p}}
\longrightarrow
\frac{Z_{m_2}^{p+j,k+1-p-j}}
{B_j^{p+j,k+1-p-j}}\,.
\end{align}
Moreover, if \(l\in I_{(p+j,k+1)}\), then
\begin{align*}
\frac{Z_j^{p,k-p}}{B_{m_1}^{p,k-p}}
\xrightarrow{\Delta_j}
\frac{Z_l^{p+j,k+1-p-j}}
{B_j^{p+j,k+1-p-j}}
\xrightarrow{\Delta_l}
\frac{Z_{m_2}^{p+j+l,k+2-p-j-l}}
{B_l^{p+j+l,k+2-p-j-l}}
\end{align*}
satisfies $\Delta_l\circ\Delta_j=0$.

\end{proposition}

\begin{definition}[Spectral complexes associated with
\((\Omega^\bullet(G),d)\)]
\label{def: spectral complexes}
For positive integers \(j,l\), define the bigraded quotient
\begin{align}\label{notation spaces}
E_{j,l}^{p,k-p}
:=
\frac{Z_j^{p,k-p}}{B_l^{p,k-p}}\,.
\end{align}
This quotient is well defined by Lemma
\ref{lem: inclusions B and Z}.

Let \(l,j,i\) be successive admissible orders, in the sense that $l\in I_{(p-l,k-1)}$, $j\in I_{(p,k)}$, and $i\in I_{(p+j,k+1)}$.
Proposition \ref{prop: mixed page differentials} then gives
\begin{align*}
E_{l,m_1}^{p-l,k-1-p+l}
\xrightarrow{\Delta_l}
E_{j,l}^{p,k-p}
\xrightarrow{\Delta_j}
E_{i,j}^{p+j,k+1-p-j}
\xrightarrow{\Delta_i}
E_{m_2,i}^{p+j+i,k+2-p-j-i},
\end{align*}
with
\begin{align*}
\Delta_j\circ\Delta_l=0\ \text{ and }
\
\Delta_i\circ\Delta_j=0,
\end{align*}
for all positive integers \(m_1,m_2\) for which the adjacent terms occur.

 Since this construction applies in every degree $k$, it yields a family of quotients $E_{j,l}^{\bullet,\bullet}$ connected by the operators $\Delta_j$ of bidegree $(j,1-j)$, hence forming a collection of complexes.
    We  will refer to the collection of such quotients as the \textit{spectral complexes} associated with the $s$-multicomplex $(\Omega^\bullet(G),d)$ and denote this collection schematically by
    \begin{align*}
        \bigg\lbrace \left( E_{j,l}^{\bullet,\bullet},\Delta_j\right)\bigg\rbrace_{j\in I_{(\bullet,\bullet)}}\,.
    \end{align*}
    \end{definition}

\begin{remark}[Quotients and orthogonal representatives]
\label{rmk: quotient versus representatives}
Unless $j=l=1$, the quotient \(E_{j,l}^{p,k-p}=Z_j^{p,k-p}/B_l^{p,k-p}\) should not be regarded directly as a subspace of \(E_1^{p,k-p}\).

If a Hilbert-space realisation is chosen in which all the relevant
subspaces $B_r^{p,q}$ are closed, then the quotient admits the
orthogonal realisation
\begin{align*}
\frac{Z_j^{p,k-p}}{B_l^{p,k-p}}
\cong
Z_j^{p,k-p}\cap
\left(B_l^{p,k-p}\right)^\perp
\subseteq
Z_1^{p,k-p}\cap
\left(B_1^{p,k-p}\right)^\perp
\cong
E_0^k\cap\Omega^{p,k-p}(G)\,.
\end{align*}
Under this identification, the differential on orthogonal representatives is obtained by applying the operator \(\Delta_j\) and then projecting onto the orthogonal complement of
\(B_j^{p+j,k+1-p-j}\).

If $Z_j^{p,k-p}$ is closed but $B_l^{p,k-p}$ is not, then
$Z_j^{p,k-p}\cap\left(B_l^{p,k-p}\right)^\perp$
realises the reduced quotient
\[
Z_j^{p,k-p}/\overline{B_l^{p,k-p}}^{\,Z_j},
\]
rather than the algebraic quotient
$Z_j^{p,k-p}/B_l^{p,k-p}$.
\end{remark}

\subsubsection{Constructing and defining spectral complexes in the Engel group}\label{subsection: spectral complexes in Engel}

As presented in Definition \ref{def: spectral complexes}, in order to construct the spectral complexes associated with the de Rham complex on the Engel group, we need to identify for which choices of weight-degree pair $(p,k)$ the quotients $E_1^{p,k-p}$ are non-trivial, and then study the relative collection of indices $I_{(p,k)}$. 

We recall the isomorphism $E_1^{p,k-p}\cong E_0^k\cap\Omega^{p,k-p}(G)$, so that we can equivalently study which spaces of Rumin forms are non-trivial in degree $k$ and weight $p$ to retrieve the same information (one can also refer to the previous diagram for a more direct way of doing this).
\begin{itemize}
    \item degree $k=0$: the only non-trivial weight-degree pair is $(0,0)$;
    \item degree $k=1$: the only non-trivial weight-degree pair is $(1,1)$;
    \item degree $k=2$: there are two non-trivial choices of weight-degree pair, namely $(3,2)$ and $(4,2)$;
    \item degree $k=3$: the only non-trivial weight-degree pair is $(6,3)$;
    \item degree $k=4$: the only non-trivial weight-degree pair is $(7,4)$.
\end{itemize}
As a direct consequence of the explicit computations carried out in Subsection \ref{subsection: Rumin complex in Engel}, we get the following collection of indices $I_{(p,k)}$:
\begin{itemize}
    \item  at $(0,0)$ we have that $d_c=d_c^1$, and so $I_{(0,0)}=\{1\}$;
    \item at $(1,1)$ we have that $d_c=d_c^2+d_c^3$, and so $I_{(1,1)}=\{2,3\}$;
    \item at $(3,2)$ we have that $d_c=d_c^3$, and so $I_{(3,2)}=\{3\}$;
    \item at $(4,2)$ we have that $d_c=d_c^2$, and so $I_{(4,2)}=\{2\}$;
    \item at $(6,3)$ we have that $d_c=d_c^1$, and so $I_{(6,3)}=\{1\}$;
    \item at $(7,4)$ we have that $d_c\equiv 0$, and so $I_{(7,4)}=\emptyset$.
\end{itemize}

Consequently, the family of spectral complexes associated with the Engel group consists of two complexes. More precisely, applying Definition~\ref{def: spectral complexes}, we obtain the following two spectral complexes:
\[\begin{tikzcd}
	{Z_{1}^{0,0}} & {Z_2^{1,0}/B_1^{1,0}} & {Z_3^{3,-1}/B_2^{3,-1}} & {Z_1^{6,-3}/B_3^{6,-3}} & {Z_1^{7,-3}/B_1^{7,-3}} \\
	{Z_1^{0,0}} & {Z_3^{1,0}/B_1^{1,0}} & {Z_2^{4,-2}/B_3^{4,-2}} & {Z_1^{6,-3}/B_2^{6,-3}} & {Z_1^{7,-3}/B_1^{7,-3}}
	\arrow["{\Delta_1}", from=1-1, to=1-2]
	\arrow["{\Delta_2}", from=1-2, to=1-3]
	\arrow["{\Delta_3}", from=1-3, to=1-4]
	\arrow["{\Delta_1}", from=1-4, to=1-5]
	\arrow["{\Delta_1}", from=2-1, to=2-2]
	\arrow["{\Delta_3}", from=2-2, to=2-3]
	\arrow["{\Delta_2}", from=2-3, to=2-4]
	\arrow["{\Delta_1}", from=2-4, to=2-5]
\end{tikzcd}\]

For the purposes of this paper, we are interested in an explicit characterisation of all the graded submodules involved in the construction of these two spectral complexes. In particular, we would to find an explicit description of each $Z_r^{p,k-p}$ and $B_r^{p,k-p}$ as subspaces of the space of smooth forms $\Omega^\bullet(G)$, but also as subspaces of the space of compactly supported smooth forms $\Omega^\bullet_c(G)$. As we will see, the distinction between the two will be significant (see for example Lemmata \ref{lemma: B esplicito nel caso liscio in grado 3} and \ref{lemma: Z3 a supporto compatto in deg 1}).
\begin{remark}[Notation for compactly supported forms]
\label{rmk: compact support convention for Z and B}
Throughout the remainder of the paper, we adopt the convention that
\[
Z_r^{p,q}\cap\Omega_c^{p+q}(G)
\ \text{and}\
B_r^{p,q}\cap\Omega_c^{p+q}(G)
\]
denote the subspaces obtained by applying Definition
\ref{def: Z and B defined} to the compactly supported subcomplex
\((\Omega_c^\bullet(G),d)\). In particular, not only the resulting
form, but also all the auxiliary witnesses appearing in the
definitions of \(Z_r^{p,q}\) and \(B_r^{p,q}\), are required to be
compactly supported.

Thus, the symbol \(\cap\Omega_c^{p+q}(G)\) is part of our notational
convention and should not be interpreted merely as the set-theoretic
intersection of \(\Omega_c^{p+q}(G)\) with the corresponding subspace
constructed inside \(\Omega^\bullet(G)\). 
\end{remark}

In order to carry out the computations, one can equivalently use Definition \ref{def: Z and B defined} or the equivalent expression in terms of Rumin forms and the Rumin differential presented in Proposition \ref{prop: Z and B in terms of d_c}. Let us first focus on arbitrary smooth forms $\Omega^\bullet(G)$.
\begin{itemize}
    \item $Z_1^{0,0} = \mathcal{C}^\infty(G)$;
    \item $Z_2^{1,0}=\ker d_c^1\cap\Omega^{1,0}(G)=\Omega^{1,0}(G) = \operatorname{span}_{\mathcal{C}^\infty(G)} \{\theta_1, \theta_2\}$ and $B_1^{1,0}$ = \{0\};
    \item $Z_3^{1,0}=\ker d_c^2\cap\Omega^{1,0}(G) = \{f \theta_1 + g \theta_2 \mid (X_2X_1-X_3)g = X_2^2f \}$ and $B_1^{1,0} =\{0\}$;
    \item $Z_3^{3,-1}=(\ker d_c^1\cap\Omega^{3,-1}(G))\cap\ker( d_c^2\vert_{\Omega^{3,-1}(G)}+d_c^1\vert_{\Omega^{4,-2}(G)})\cap\Omega^{3,-1}(G)+B_1^{3,-1}=\Omega^{3,-1}(G)$ and $B_2^{3,-1}=\operatorname{Im}d_0\cap\Omega^{3,-1}(G)=B_1^{3,-1}=\operatorname{span}_{\mathcal{C}^\infty(G)}\{\theta_1\wedge\theta_3\}$;
    \item $Z_2^{4,-2}=\ker d_c^1\cap\Omega^{4,-2}(G)=\operatorname{span}_{\mathcal{C}^\infty(G)}\{\theta_1\wedge\theta_4\}$ and $B_3^{4,-2}=\{0\}$;
    \item $Z_1^{6,-3} =\ker d_0\cap\Omega^{6,-3}(G)= \operatorname{span}_{\mathcal{C}^\infty(G)}\{\theta_1 \wedge \theta_3 \wedge \theta_4, \theta_2 \wedge \theta_3 \wedge \theta_4\}$;
    \item $B_3^{6,-3} = \operatorname{Im}(d_c^2\vert_{\Omega^{4,-2}(G)})\cap\Omega^{6,-3}(G)=\{(X_3f + X_1X_2f) \theta_1 \wedge \theta_3 \wedge \theta_4 + X_2^2f \theta_2 \wedge \theta_3 \wedge \theta_4 \mid f \in \mathcal{C}^\infty(G)\}$ and $B_2^{6,-3}=\{0\}$;
    \item $Z_1^{7,-3}=\Omega^{7,-3}(G)=\operatorname{span}_{\mathcal{C}^\infty(G)}\{\theta_1\wedge\theta_2\wedge\theta_3\wedge\theta_4\}$ and $B_1^{7,-3}=\{0\}$.
\end{itemize}
Under the canonical identifications of the relevant quotients with
their Rumin representatives, the two complexes take the following form:
\[\begin{tikzcd}
	\mathcal{C}^\infty(G) & {E_0^1} & {E_0^2\cap\Omega^{3,-1}(G)} & {E_0^3/B_3^{6,-3}} & {\Omega^{7,-3}(G)} \\
{\mathcal{C}^\infty(G)} & {Z_3^{1,0}} & {E_0^2\cap\Omega^{4,-2}(G)} & {E_0^3} & {\Omega^{7,-3}(G)}
	\arrow["{d_c^1}", from=1-1, to=1-2]
	\arrow["{d_c^2}", from=1-2, to=1-3]
	\arrow["{\Delta_3}", from=1-3, to=1-4]
	\arrow["{\Delta_1}", from=1-4, to=1-5]
	\arrow["{d_c^1}", from=2-1, to=2-2]
	\arrow["{d_c^3}", from=2-2, to=2-3]
	\arrow["{d_c^2}", from=2-3, to=2-4]
	\arrow["{d_c^1}", from=2-4, to=2-5]
\end{tikzcd}\]
In other words, after decomposing the Rumin complex according to weight, only
two terms involve an additional restriction or quotient, namely
\begin{align*}
    E_0^2\cap\Omega^{3,-1}(G)\xrightarrow[]{\Delta_3}E_0^3/B_3^{6,-3}\ \text{ and }\ Z_3^{1,0}\xrightarrow[]{d_c^3} E_0^2\cap\Omega^{4,-2}(G)\,.
\end{align*}

\begin{lemma}\label{lemma: B esplicito nel caso liscio in grado 3}
    We have the following equality
    \begin{equation*}
        \{ (-X_3f - X_1X_2f, -X_2^2 f) 
        \mid f \in \C^\infty(G)\} = \{ (g,h) \in \C^\infty(G) \times \C^\infty(G) \mid X_2 g = X_1 h\}
    \end{equation*}
\end{lemma}

\begin{proof}
We first prove the inclusion from left to right. Let
\[
g=-X_3f-X_1X_2f\ ,
\
h=-X_2^2f\,.
\]
Since \([X_2,X_3]=0\) and
\(
X_2X_1=X_1X_2-X_3
\),
we have
\begin{align*}
X_2g
&=-X_2X_3f-X_2X_1X_2f=-X_3X_2f-(X_1X_2-X_3)X_2f\\
&=-X_3X_2f-X_1X_2^2f+X_3X_2f
=-X_1X_2^2f
=X_1h.
\end{align*}

For the converse inclusion, we use global coordinates adapted to \(X_2\) and
\(X_3\). In the standard global coordinates \((x,y,z,w)\in\mathbb{R}^4\) (exponential coordinates of the second type, see for example \cite{LeDonneTripaldi2021})
, the left-invariant vector fields can be
written as
\[
X_1=\partial_x\ ,
\ 
X_2=\partial_y+x\partial_z+\frac{x^2}{2}\partial_w\ ,
\
X_3=\partial_z+x\partial_w\ ,
\
X_4=\partial_w\,.
\]
Introduce new global coordinates by setting $a=x$, $s=y$, $t=z-xy$, and $r=w-xz+x^2y/2$. The inverse coordinate change is
$x=a$, $y=s$, $
z=t+as$ and $
w=r+at+\frac{a^2}{2}s$,
and hence this is a global polynomial change of coordinates.
In these new coordinates, we have the following expression for our left-invariant vector fields
\[
X_1=\partial_a-s\partial_t-t\partial_r\ ,\ 
X_2=\partial_s\ ,\
X_3=\partial_t\ ,\ 
X_4=\partial_r\,.
\]
Now let \(g,h\in C^\infty(G)\) satisfy
\[
X_2g=X_1h.
\]
Since $X_2=\partial_s$, we are able to construct a smooth function \(f_0\) satisfying
\(
-X_2^2f_0=h\), by imposing
\[
f_0(a,s,t,r)
=
-\int_0^s (s-\sigma)
h(a,\sigma,t,r)\,d\sigma\,.
\]
Indeed, differentiating twice with respect to \(s\) gives
\(
X_2^2f_0=\partial_s^2f_0=-h\),
and hence
$-X_2^2f_0=h$.

Set
\[
g_0=-X_3f_0-X_1X_2f_0\,.
\]
Following similar simplifications as in the first part of the proof, we have $X_2g_0=X_1(-X_2^2f_0)=X_1h$, and so
\[
X_2(g-g_0)=X_2g-X_2g_0=X_1h-X_1h=0\,.
\]
Since \(X_2=\partial_s\), we know that the function \(g-g_0\) is independent of \(s\), and thus there is a smooth function \(\rho=\rho(a,t,r)\) such that $(g-g_0)(a,s,t,r)=\rho(a,t,r)$.

Consequently, the function $k$ defined as
\[
k(a,s,t,r)
=
-\int_0^t \rho(a,\tau,r)\,d\tau
\] is also independent of \(s\), and therefore $X_2k=0$, and hence also $-X_2^2k=0$.
Moreover, since \(X_3=\partial_t\),
we have that $X_3k=-\rho$, 
so that
\[
-X_3k-X_1X_2k
=
-X_3k
=
\rho
=
g-g_0.
\]

Finally, set
\[
f=f_0+k.
\]
Then $
-X_2^2f
=
-X_2^2f_0-X_2^2k
=
h$,
and
\begin{align*}
-X_3f-X_1X_2f
&=
\bigl(-X_3f_0-X_1X_2f_0\bigr)
+
\bigl(-X_3k-X_1X_2k\bigr)=g_0+g-g_0=g\,.
\end{align*}
Therefore
\[
(g,h)
=
\bigl(-X_3f-X_1X_2f,-X_2^2f\bigr)\,,
\]
which proves the second inclusion.
\end{proof}

\begin{remark}
The preceding equality is specific to unrestricted smooth coefficients
on the whole group \(G\). If we instead consider the corresponding spaces obtained by requiring all coefficients
to be compactly supported, then
\[
B_{3}^{6,-3}\cap\Omega_c^{6,-3}(G)\subsetneq Z_{2}^{6,-3}\cap\Omega_c^{6,-3}(G).
\]
Indeed, in the global coordinates \((a,s,t,r)\) introduced in the proof,
we have \(X_2=\partial_s\). Thus, if
\[
(g,h)
=
\bigl(-X_3f-X_1X_2f,-X_2^2f\bigr)
\]
for some \(f\in C_c^\infty(G)\), then $h=-\partial_s^2f$.
Consequently, for every fixed \((a,t,r)\),
\[
\int_{\mathbb R}h(a,s,t,r)\,ds=0
\]
and, in fact,
\[
\int_{\mathbb R}s\,h(a,s,t,r)\,ds=0.
\]
These integral conditions along the $X_2$-lines are not consequences of the
compatibility equation \(X_2g=X_1h\).

For unrestricted smooth coefficients, under the identifications
\[
B_3^{6,-3}
=
\left\{
\bigl(-X_3f-X_1X_2f,-X_2^2f\bigr)
:
f\in C^\infty(G)
\right\}
\]
and
\[
Z_2^{6,-3}
=
\left\{
(g,h)\in C^\infty(G)^2
:
X_2g=X_1h
\right\},
\]
the proposition gives
\[
B_3^{6,-3}=Z_2^{6,-3}.
\]
Consider now the relevant part of the spectral complex,
\[
\frac{Z_3^{3,-1}}{B_2^{3,-1}}
\xrightarrow{\Delta_3}
\frac{Z_1^{6,-3}}{B_3^{6,-3}}
\xrightarrow{\Delta_1}
\frac{Z_1^{7,-3}}{B_1^{7,-3}}.
\]
Since consecutive differentials in the spectral complex compose to
zero, we have $\operatorname{Im}\Delta_3
\subseteq
\ker\Delta_1$.
Moreover,
\[
\ker\left(
\Delta_1:
\frac{Z_1^{6,-3}}{B_3^{6,-3}}
\longrightarrow
\frac{Z_1^{7,-3}}{B_1^{7,-3}}
\right)
=
\frac{Z_2^{6,-3}}{B_3^{6,-3}}.
\]
The equality \(B_3^{6,-3}=Z_2^{6,-3}\) therefore gives
\[
\ker\Delta_1
=
\frac{Z_2^{6,-3}}{B_3^{6,-3}}
=
0,
\]
and hence
\[
\Delta_3=0.
\]

Thus, in the global smooth setting, this portion of the spectral
complex takes the form
\[
\frac{Z_3^{3,-1}}{B_2^{3,-1}}
\xrightarrow{\,0\,}
\frac{Z_1^{6,-3}}{Z_2^{6,-3}}
\xrightarrow{\Delta_1}
\frac{Z_1^{7,-3}}{B_1^{7,-3}},
\]
where the second map is injective. Notice that this does not mean that
\(Z_1^{6,-3}/B_3^{6,-3}\) is the zero space. Rather, it means that its
only class annihilated by the following \(\Delta_1\) is the zero class.

Equivalently, although a differential expression representing
\(\Delta_3[\alpha]\) may be nonzero as a smooth differential form, it
belongs to \(B_3^{6,-3}\) and therefore represents the zero class in
\(Z_1^{6,-3}/B_3^{6,-3}\).
\end{remark}

\begin{lemma}\label{lemma: Z3 a supporto compatto in deg 1}
The space $Z_3^{1,0}$ when restricted to compactly supported forms $\Omega_c^1(G)$ simplifies as follows
    \begin{equation*}
        Z_{3}^{1,0} \cap \Omega_c^1(G) = \left\{ X_1 f \theta_1 + X_2 f\theta_2 \,\, \big | \,\, f \in \mathcal{C}_c^\infty(G)\right\} = d_1 \left( \mathcal{C}_c^\infty(G) \right)
    \end{equation*}
\end{lemma}

\begin{proof}

The first inclusion
\[
d_1\bigl(\C_c^\infty(G)\bigr)
\subseteq Z_3^{1,0}\cap\Omega_c^1(G)
\]
follows directly from the fact that $Z_3^{1,0}=\ker d_c^2\cap\Omega^{1,0}(G)$, $d_c^2\circ d_c^1=0$ when applied to any smooth function, and that $d_c^1=d_1$ as operators on smooth functions. Since the statement holds for any smooth function, it also holds in the case of compactly supported smooth functions, since differentiation preserves compact support.

For the converse inclusion, let $\alpha=f_1\theta_1+f_2\theta_2
\in Z_3^{1,0}\cap\Omega_c^1(G)$, so that \(f_1,f_2\in \C_c^\infty(G)\) satisfying $(X_2X_1-X_3)f_2=X_2^2f_1$.
We use the global coordinates \((a,s,t,r)\) introduced previously, in
which $X_2=\partial_s$ and $X_3=\partial_t$. 
For every \((a,t,r)\), define
\[
F(a,t,r)
:=
\int_{\mathbb R}f_2(a,s,t,r)\,ds.
\]
Since \(f_2\) is compactly supported, \(F\) is a smooth compactly
supported function of \((a,t,r)\). Moreover,
\begin{align*}
\partial_tF(a,t,r)
&=
\int_{\mathbb R}\partial_tf_2(a,s,t,r)\,ds=
\int_{\mathbb R}X_3f_2(a,s,t,r)\,ds\\=&
\int_{\mathbb R}X_2(X_1f_2-X_2f_1)(a,s,t,r)\,ds=
\int_{\mathbb R}\partial_s(X_1f_2-X_2f_1)(a,s,t,r)\,ds
=
0.
\end{align*}
Hence \(F\) is a compactly supported function of $(a,t,r)$ which is also independent of \(t\), which implies that $F\equiv 0$.

Define now
\[
\varphi(a,s,t,r)
:=
\int_{-\infty}^{s}f_2(a,\sigma,t,r)\,d\sigma\,,
\]
then $X_2\varphi=\partial_s\varphi=f_2$, with \(\varphi\in C_c^\infty(G)\). Indeed, \(\varphi\) vanishes
below the \(s\)-support of \(f_2\), while above the \(s\)-support of
\(f_2\) one has
\[
\varphi(a,s,t,r)
=
\int_{\mathbb R}f_2(a,\sigma,t,r)\,d\sigma
=
F(a,t,r)
=
0.
\]
Compactness in the remaining variables follows directly from the
compact support of \(f_2\).

We next show that
\[
X_1f_2-X_2f_1=X_3\varphi.
\]
Using \(X_2(X_1f_2-X_2f_1)=X_3f_2\), \(X_2\varphi=f_2\), and
\([X_2,X_3]=0\), we obtain
\begin{align*}
X_2\bigl(X_1f_2-X_2f_1-X_3\varphi\bigr)
&=
X_2X_1f_2-X_2^2f_1-X_2X_3\varphi=
X_3f_2-X_3X_2\varphi
=
0.
\end{align*}
Thus \(X_1f_2-X_2f_1-X_3\varphi\) is independent of \(s\). On the other hand, it
is compactly supported, since both $X_1f_2-X_2f_1$ and \(\varphi\) are compactly
supported. A compactly supported function that is independent of \(s\)
must vanish identically, and so $X_1f_2-X_2f_1=X_3\varphi$.

Finally, consider
\[
q:=f_1-X_1\varphi.
\]
Then $X_2q=X_2f_1-X_2X_1\varphi=
X_2f_1-X_1X_2\varphi+X_3\varphi=
X_2f_1-X_1f_2+X_1f_2-X_2f_1=0$, which means that \(q\) is independent of \(s\).
Since \(q\) is again compactly supported, it follows that \(q=0\).
Consequently, $f_1=X_1\varphi$.

We have therefore proved that
\[
\alpha
=
X_1\varphi\,\theta_1+X_2\varphi\,\theta_2
=
d_1\varphi
\]
for some \(\varphi\in C_c^\infty(G)\), completing the proof.
\end{proof}
\begin{remark}
The preceding lemma may alternatively be viewed as a consequence
of the injectivity of \(d_c^2\) on compactly supported Rumin
2-forms of weight \(4\), together with the vanishing of the
compactly supported de Rham cohomology in degree \(1\).

Let us first observe that the restriction of \(d_c^2\) to compactly
supported Rumin 2-forms of weight \(4\) has trivial kernel. Since 
\[
\frac{Z_{2}^{4,-2}}{B_{3}^{4,-2}}
=
E_0^2\cap\Omega^{4,-2}(G)\bigr),
\]
every compactly supported element is represented by a form $\beta=f\theta_1\wedge\theta_4$ with $f\in\C_c^\infty(G)$. If \(d_c^2\beta=0\), then $X_2^2f=0$, and in the global coordinates \((a,s,t,r)\) used above, \(X_2=\partial_s\),
and hence $\partial_s^2f=0$.

Thus, for every fixed \((a,t,r)\), there exist smooth functions
\(A=A(a,t,r)\) and \(B=B(a,t,r)\) such that
\[
f(a,s,t,r)=sA(a,t,r)+B(a,t,r).
\]
Since \(f\) is compactly supported in the \(s\)-variable, necessarily
\(A=B=0\). Therefore \(f=0\), and hence
\[ \ker \left(\Delta_2\colon \frac{Z_2^{4,-2}\cap\Omega_c^2(G)}{B_3^{4,-2}\cap\Omega_c^2(G)}\longrightarrow \frac{Z_1^{6,-3}\cap\Omega_c^3(G)}{B_2^{6,-3}\cap\Omega_c^3(G)}\right)
=\ker\left(
d_c^2\colon E_0^2\cap\Omega_c^{4,-2}
\longrightarrow
E_0^3\cap\Omega_c^3(G)
\right)
=
\{0\}\,.
\]

The relevant portion of the second spectral complex is
\[
\frac{Z_{3}^{1,0}\cap\Omega^1_c(G)}{B_{1}^{1,0}\cap\Omega^{1}_c(G)}
\xrightarrow{\Delta_3}
\frac{Z_{2}^{4,-2}\cap\Omega^{2}_c(G)}{B_{3}^{4,-2}\cap\Omega_c^2(G)}
\xrightarrow{\Delta_2}
\frac{Z_{1}^{6,-3}\cap\Omega_c^3(G)}{B_{2}^{6,-3}\cap\Omega_c^3(G)}.
\]
Since this is a complex, $\Delta_2\circ\Delta_3=0$, and so the injectivity of \(\Delta_2\)  implies $\Delta_3$ must be the zero map.

The operators involved in the construction of the Rumin and spectral
complexes preserve compact support. Hence the compactly supported
version of the second spectral complex carries the compactly supported
de Rham cohomology in degree \(1\). Therefore, using the fact that the Engel group is diffeomorphic to \(\mathbb R^4\), we have
\begin{align*}
0=H_c^1(G)
&\simeq
\frac{\ker\Delta_3}{\operatorname{Im}\Delta_1}\simeq
\frac{Z_{3}^{1,0}\cap\Omega_c^1(G)}{B_{2}^{1,0}\cap\Omega^1_c(G)},
\end{align*}
where we used \(\Delta_3=0\) and
\[
\operatorname{Im}\Delta_1
=
\frac{B_{2}^{1,0}\cap\Omega^1_c(G)}{B_{1}^{1,0}\cap\Omega_c^1(G)}\,.
\]
It then follows that $Z_{3}^{1,0}\cap\Omega_c^1(G)=B_{2}^{1,0}\cap\Omega_c^1(G)$, 
or equivalently $Z_3^{1,0}\cap\Omega_c^1(G)
=
d_1\bigl(\C_c^\infty(G)\bigr)$.

\end{remark}

\section{Sobolev mappings and Pansu derivative}

In this section, we recall the notion of Sobolev mappings between Carnot groups and some of their basic properties, in particular almost everywhere Pansu differentiability (often shortened as $P$-differentiability).
Let $G$ be an arbitrary Carnot group, $U\subset G$ an open set, and $(X_1, \ldots X_m)$ an orthonormal basis of the horizontal layer $V_1$ of its Lie algebra $\mathfrak g$.

\begin{definition}[Scalar-valued Sobolev functions]\label{def: Sobolev for R}
 Let $1\le q\le \infty$. A function $\varphi \colon U \rightarrow \mathbb{R}$
belongs to $W^{1,q}(U)$ if $\varphi\in L^q(U)$ and its distributional derivatives $X_i\varphi$ belong to $L^q(U)$ for every $i=1,\ldots, m$.

We say that $\varphi\in W^{1,q}_{\mathrm{loc}}(U)$ if
$\varphi|_{U'}\in W^{1,q}(U')$ for every open set
$U'\Subset U$.
\end{definition}

Definition \ref{def: Sobolev for R} admits a natural extension to maps taking values in a separable metric space, and in particular to maps between Carnot groups. Several equivalent characterizations of the space of Sobolev maps between Carnot groups are available; see, for instance, \cite[Proposition 4.2]{Vodopyanov2000}. We recall here the metric-valued formulation adopted by Vodop\`yanov, which,  when the domain is Euclidean, reduces to Reshetnyak's definition \cite{Reshetnyak1997}. Let $\widetilde G$ be another Carnot group endowed with a homogeneous distance $d_{\widetilde G}$. Denote by $\widetilde{V}_1$ its horizontal layer.

\begin{definition}
    Let $1\le q\le \infty$. We say that a measurable map $\varphi: U \rightarrow \widetilde G$ belongs to the Sobolev space $W^{1,q}_{\mathrm{loc}}(U, \widetilde G)$ if the following hold:
    \begin{itemize}
        \item for every $z \in \widetilde G$, the function $\varphi_z : U \rightarrow \mathbb R$, defined as
        \begin{equation*}
            \varphi_z(x) = d_{\widetilde G} (\varphi(x), z)
        \end{equation*}
        is in $W^{1,q}_{\mathrm{loc}}(U)$;
        \item there exists a non-negative function $g \in L^q_{\mathrm{loc}}(U)$ such that for every $z \in \widetilde G$,  $|\nabla_{G} \varphi_z (x)| \le g(x)$ for a.e. $x \in U$, where $\nabla_{G} \varphi_z = \sum_{i=1}^m (X_i\varphi_z) X_i$ denotes the horizontal distributional gradient of $\varphi_z$.
    \end{itemize}
\end{definition}

An equivalent characterization of Sobolev mappings is given in terms of
absolute continuity along horizontal lines. More precisely, after modifying $\varphi$ on a set of measure zero, for each
$i=1,\ldots,m$ the map $\varphi$ is locally absolutely continuous
along almost every horizontal line $\gamma(t)=x\exp_G(tX_i)$.

At almost every point on such a line, its left-trivialised derivative
is horizontal (see \cite[Proposition 4.2]{Vodopyanov2000})
 and we denote it by \[
X_i\varphi(x)
:=
(dL_{\varphi(x)^{-1}})_{\varphi(x)}
\left.
\frac{d}{dt}
\right|_{t=0}
\varphi\bigl(x\exp_G(tX_i)\bigr)
\in\widetilde V_1\,.
\]
Moreover, by \cite[Proposition 4.3]{Vodopyanov2000},
\begin{align*}
\exp_{\widetilde G}\bigl(X_i\varphi(x)\bigr)
=
\lim_{t\to 0}
\widetilde\delta_{1/t}
\left(
\varphi(x)^{-1}
\varphi\bigl(x\exp_G(tX_i)\bigr)
\right).
\end{align*}
The vectors $X_i\varphi(x)$ determine the formal horizontal differential $$D_h \varphi (x) \colon V_1 \rightarrow \widetilde{V}_1\ \text{ defined by } D_h\varphi(x)X_i = X_i \varphi(x)\ \text{ for } i=1,\ldots,m\,.$$
A non-trivial consequence of Sobolev regularity is that, at almost every
point, this horizontal linear map is compatible with the Lie algebra
relations. More precisely, by
\cite[Proposition 4.4]{Vodopyanov2000}, if $\varphi\in W^{1,q}_{\mathrm{loc}}(U,\widetilde G)$, then $\varphi$ is approximately $P$-differentiable at almost every
$x\in U$. 
Thus, for almost every
$x$, there exists a contact Lie group homomorphism $\phi_x\colon G \longrightarrow \widetilde G$ such that the rescaled maps
\begin{equation}\label{eq: rescaling maps Pansu}
        \varphi_{x,t}(y) := \Tilde{\delta}_t^{-1}\left(\varphi(x)^{-1} \varphi (x \delta_t y)\right).
    \end{equation}
converge locally in measure to $\phi_x$ as $t \rightarrow 0$. At every such point of approximate $P$-differentiability, the Lie algebra homomorphism 
\begin{equation*}
    D\varphi(x)
:=
\exp_{\widetilde G}^{-1}
\circ \phi_x
\circ \exp_G
\colon
\mathfrak g\longrightarrow\widetilde{\mathfrak{g}}
\end{equation*}
is determined by the formal horizontal differential $D_h\varphi(x)$.

When the convergence of the rescaled maps is locally uniform, one obtains the usual notion of Pansu differentiability. To be more precise, denote by $\delta_t$ and $\Tilde{\delta}_t$ the families of dilations on $G$ and $\widetilde G$, respectively.

\begin{definition}[Pansu differentiability]
We say that $\varphi$ is Pansu differentiable at $x\in U$ if there
exists a homogeneous (i.e. it commutes with dilations) Lie group homomorphism
$\phi_x\colon G\to\widetilde G$ such that
\[
\varphi_{x,t}\longrightarrow\phi_x
\ \text{ locally uniformly on $G$ as }t\to 0\,.
\]
Equivalently, for every compact set $K\subset G$,
\[
\sup_{y\in K}
d_{\widetilde G}\bigl(\varphi_{x,t}(y),\phi_x(y)\bigr)
\xrightarrow[t\to 0]{}0\,.
\]
The homomorphism $\phi_x$ is called the Pansu derivative of $\varphi$
at $x$ and is denoted by $D_P\varphi(x)$.
\end{definition}

When the Sobolev exponent
is larger than the homogeneous dimension of the domain, approximate $P$-differentiability can be
strengthened to ordinary Pansu differentiability. More precisely, the
following holds.

\begin{theorem}[{\cite[Corollary 4.1]{Vodopyanov2000}}]\label{thm: Vodopyanov Pansu diff}
    Let $\varphi \colon U \rightarrow \widetilde G$ be a map in $W^{1,q}_{\mathrm{loc}}(U, \widetilde G)$ and let $Q$ denote the homogeneous dimension of $G$. If $q>Q$, then
$\varphi$ is Pansu differentiable at almost every point $x\in U$.
Moreover, 
\begin{equation*}
    D_P\varphi(x)
=
\exp_{\widetilde G}
\circ D\varphi(x)
\circ\exp_G^{-1},
\end{equation*}
where $D\varphi(x)\colon\mathfrak g\longrightarrow \widetilde{\mathfrak{g}}$ is the contact Lie algebra homomorphism determined by $D_h\varphi(x)$.
\end{theorem}

\medskip
We now specialise the preceding discussion to maps between Engel groups and describe explicitly the structure of the Pansu differential and the algebraic restrictions imposed by the homomorphism property.

Let $G$ and $\widetilde G$ be Engel groups, with stratified Lie algebras generated by adapted bases $\{X_1, X_2, X_3, X_4\}$ and $\{\widetilde{X}_1, \widetilde{X}_2, \widetilde{X}_3, \widetilde{X}_4\}$ satisfying
\begin{align*}
    [X_1,X_2]=X_3\, ,\, [X_1,X_3]=X_4\,,
\end{align*}
and
\begin{align*}
    [\widetilde{X}_1,\widetilde{X}_2]=\widetilde{X}_3\, ,\, [\widetilde{X}_1,\widetilde{X}_3]=\widetilde{X}_4\,.
\end{align*}

Let $\varphi: U \rightarrow \widetilde G$ belong to $W^{1,q}_{\mathrm{loc}}(U, \widetilde G)$ for $q> \dim_H G = 7$. We identify $G$ with $\mathbb{R}^4$ through the exponential coordinates of the first kind, 
\begin{equation*}
    (x_1,x_2,x_3,x_4) \longmapsto \exp_G\Big(\sum_{i=1}^4 x_i X_i\Big)\,.
\end{equation*}
Accordingly, we write 
\begin{equation*}
    \varphi(x) = \exp_{\widetilde G} \Big( \sum_{j=1}^4 \varphi_j(x)\widetilde{X}_j \Big)\,.
\end{equation*}
By Theorem \ref{thm: Vodopyanov Pansu diff}, the map $\varphi$ is Pansu differentiable at almost every $x\in U$. At every such point, the induced contact Lie algebra homomorphism 
\begin{equation*}
    D\varphi(x)= \exp^{-1}_{\widetilde G} \circ D_P\varphi(x)\circ \exp_G\colon \mathfrak g \longrightarrow \widetilde{\mathfrak g}
\end{equation*}
restricts on the horizontal layer to the formal horizontal differential 
\begin{equation*}
    D_h\varphi(x)\colon V_1 \longrightarrow \widetilde V_1\,.
\end{equation*}
With respect to the bases $(X_1, X_2)$ and $(\widetilde{X}_1, \widetilde{X}_2)$
the matrix representing $D_h\varphi(x)$ is

\begin{equation*}
D_h\varphi(x) =
    \begin{pmatrix}
    X_1\varphi_1(x) & X_2\varphi_1(x) \\
        X_1\varphi_2(x) & X_2\varphi_2(x) 
\end{pmatrix}.
\end{equation*}
Since $D\varphi(x)$ is a Lie algebra homomorphism, its action on the higher layers is determined by its restriction to $V_1$. Indeed,
\begin{align*}
    D\varphi(x) (X_3) &= [D\varphi(x) (X_1), D\varphi(x) (X_2)] = [D_h\varphi(x) (X_1), D_h\varphi(x) (X_2)] = \det (D_h\varphi(x)) \widetilde{X}_3\,,\\
    D\varphi(x) (X_4) &= [D\varphi(x) (X_1), D\varphi(x) (X_3)] = [D_h\varphi(x) (X_1), D\varphi(x) (X_3)] = X_1 \varphi_1(x) \det (D_h\varphi(x)) \widetilde{X}_4\,.
\end{align*}
Therefore, with respect to the adapted bases $(X_1,X_2, X_3, X_4)$ and $(\widetilde{X}_1,\widetilde{X}_2, \widetilde{X}_3, \widetilde{X}_4)$, the matrix representing $D\varphi(x)$ is
\begin{equation*}
    \begin{pmatrix}
    X_1\varphi_1(x) & X_2\varphi_1(x) & 0 & 0  \\
        X_1\varphi_2(x) & X_2\varphi_2(x) & 0 & 0\\ 0 & 0 & \det(D_h\varphi(x)) & 0 \\
        0 & 0 & 0 & X_1 \varphi_1(x) \det (D_h\varphi(x))        
\end{pmatrix}.
\end{equation*}
The homomorphism property also imposes a compatibility condition on the horizontal differential coming from the vanishing bracket relation $[X_2,X_3] = 0$. Indeed, 
\begin{equation*}
    0 = D\varphi( [X_2, X_3] ) = [D\varphi(x) (X_2), D\varphi(x) (X_3)] = X_2\varphi_1 (x) \det (D_h\varphi(x)) \widetilde{X}_4 \,.
\end{equation*}
It follows that, for almost every $x \in U$,
\begin{equation*}
    X_2\varphi_1(x)=0 \quad\text{or}\quad \det(D_h\varphi(x))=0. 
\end{equation*}

Thus, although the formal horizontal differential $D_h\varphi(x):V_1\longrightarrow\widetilde V_1$ is a linear map between the horizontal layers, it cannot be arbitrary. The requirement that it extend to a Lie algebra homomorphism forces it to preserve not only the non-vanishing bracket relations defining the higher layers, but also the vanishing bracket relations of the Engel algebra.

\section{Spectral currents on  the Engel group}
In the classical de Rham theory, currents are defined as continuous linear functionals on compactly supported smooth forms endowed with their natural LF topology. The exterior differential extends to currents by duality, and every smooth differential form naturally defines a current by integration. In this way, after the usual reversal of degrees, the de Rham complex embeds into the corresponding complex of currents.

Our aim in this section is to develop an analogous construction for the spectral complexes introduced above. We focus on the case of the Engel group to show that the spaces of compactly supported spectral forms carry natural locally convex topologies, that their continuous duals inherit the transposed spectral differentials, and that the spectral complexes embed naturally into these complexes of currents.

Let $G$ be the Engel group where $n =4$ and $Q=7$ are the topological and homogeneous dimensions of $G$, respectively.

We begin with the spectral complexes constructed from compactly supported smooth forms. We use throughout the convention introduced in Remark \ref{rmk: compact support convention for Z and B}: 
\[\text{either }\ 
Z_r^{p,k-p}(G)\cap\Omega_c^k(G)
\ \text{ or }\
B_l^{p,k-p}(G)\cap\Omega_c^k(G)
\]
means that the defining equations, including all auxiliary forms appearing in them, are considered entirely within the compactly supported de Rham complex.

Let us define the space of compactly supported spectral forms by
\begin{align}\label{eq: compactly supported spectral forms}
\mathscr D_{r,l}^{p,k-p}(G)
:=
\big(Z_{r}^{p,k-p}(G)\cap\Omega_c^k(G)\big)\big/
\big(B_{l}^{p,k-p}(G)\cap\Omega_c^k(G)\big).
\end{align}
We need to introduce a topology on the numerator of
\eqref{eq: compactly supported spectral forms}. Let \(K\subset G\) be compact and set
\begin{align*}
\Omega_K^{p,k-p}(G)
:=
\{
\omega\in\Omega_c^{p,k-p}(G)\mid
\operatorname{spt}\omega\subseteq K
\}\,.
\end{align*}
The scalar product fixed in the construction of the Rumin complex (see Subsection \ref{subsection: scalar product}) induces a norm, denoted by \(\lvert\cdot\rvert\), on each exterior power
\(\bigwedge^k\mathfrak g^\ast\).

For \(N\in\mathbb N\), define
\begin{align}\label{eq: seminorms fixed support forms}
\mathfrak p_{K,N}(\omega)
:=
\max_{\substack{0\leq s\leq N\\
i_1,\ldots,i_s\in\{1,2,3,4\}}}
\sup_{x\in K}
\left|
X_{i_1}\cdots X_{i_s}\omega(x)
\right|\,,
\end{align}
where, for \(s=0\), the corresponding operator is understood to be the identity. The left-invariant vector fields \(X_i\)s act on the coefficient functions of \(\omega\) with respect to the left-invariant basis of forms. More explicitly, if $\omega=\sum_{I} \omega_I\theta_I$
with respect to the left-invariant basis of forms, then
\[
X_{i_1}\cdots X_{i_s}\omega
:=
\sum_I
X_{i_1}\cdots X_{i_s}(\omega_I)\,\theta_I\,.
\]

Since $(X_1,X_2,X_3,X_4)$ form a global smooth frame of \(TG\), the seminorms
$\{\mathfrak p_{K,N}\}_{N\in\mathbb N}$
define the usual Fréchet topology on
\(\Omega_K^{p,k-p}(G)\).
We then set
\begin{align*}
Z_r^{p,k-p}(G;K)
:=\big(
Z_{r}^{p,k-p}(G)\cap\Omega^k_c(G)\big)
\cap\Omega_K^{p,k-p}(G)\,.
\end{align*}

\begin{lemma}\label{lem: Z fixed support Frechet}
For every compact set \(K\subset G\) and every space
\(Z_r^{p,k-p}(G)\) occurring in the spectral complexes of the Engel group, the space
\(Z_r^{p,k-p}(G;K)\), endowed with the restrictions of the seminorms
\(\{\mathfrak p_{K,N}\}_{N\in\mathbb N}\), is a Fr\'echet space.
\end{lemma}
\begin{proof}
The space $\Omega_K^{p,k-p}(G)$ is a closed linear subspace of the Fr\'echet space $\Omega_K^k(G)$ and is therefore itself a Fr\'echet space. From the explicit descriptions in Subsection \ref{subsection: spectral complexes in Engel}, we have that for every choice of $(r,p,k)$ occurring in the spectral complexes of the Engel group, there exists a continuous linear differential operator of finite order $P_{r,p,k}$ from $\Omega_K^{p,k-p}(G)$ 
into a finite product of fixed-support spaces of smooth forms such that
\[
Z_r^{p,k-p}(G;K)=\ker P_{r,p,k}\,.
\]
Here $P_{r,p,k}$ may be the zero operator or a zero-order coefficient projection. Since differential operators are continuous with respect to the usual Fr\'echet topology on fixed-support smooth forms, $Z_r^{p,k-p}(G;K)$ is closed in $\Omega_K^{p,k-p}(G)$. Hence it is a Fr\'echet space.
\end{proof}
We endow $Z_{r}^{p,k-p}(G)\cap\Omega_c^k(G)$ with the locally convex subspace
topology inherited from $\Omega_c^{p,k-p}(G)$, equipped with its usual
LF-topology. On each fixed-support subspace
$Z_r^{p,k-p}(G;K)$, this topology restricts to the Fr\'echet topology
described above.

We then endow $
\mathscr D_{r,l}^{p,k-p}(G)$
with the corresponding quotient topology. This is a locally convex
topology, although it need not be Hausdorff.

\begin{remark}\label{rmk: closedness of B and Hausdorff quotient}
The quotient topology on
\(\mathscr D_{r,l}^{p,k-p}(G)\) is Hausdorff if and only if
\(B_{l}^{p,k-p}(G)\cap\Omega_c^k(G)\) is closed in
\(Z_{r}^{p,k-p}(G)\cap\Omega_c^k(G)\). Thus, the closedness of the spaces
\(Z_r^{p,k-p}(G;K)\), established in Lemma
\ref{lem: Z fixed support Frechet}, does not by itself imply that the spectral quotient is Hausdorff.

Regardless of closedness, the continuous dual of the quotient is naturally identified with the annihilator of \(B_{l}^{p,k-p}(G)\cap\Omega_c^k(G)\):
\begin{align}\label{eq: dual quotient as annihilator}
\left(\mathscr D_{r,l}^{p,k-p}(G)\right)'
\cong
\{
T\in\left(Z_{r}^{p,k-p}(G)\cap\Omega_c^k(G)\right)' \mid
T(\beta)=0
\text{ for every }
\beta\in B_{l}^{p,k-p}(G)\cap\Omega_c^k(G)
\}\,.
\end{align}
Here the prime denotes the continuous dual.

If \(B_{l}^{p,k-p}(G)\cap\Omega_c^k(G)\) is not closed, this dual depends only on its closure and therefore coincides with the dual of the corresponding Hausdorff quotient
\[
\big(Z_{r}^{p,k-p}(G)\cap\Omega_c^k(G)\big)\big/
\overline{\big(B_{l}^{p,k-p}(G)\cap\Omega_c^k(G)\big)}\,.
\]
\end{remark}

\begin{remark}
The proof of Lemma \ref{lem: Z fixed support Frechet} relies on the
explicit descriptions of the spaces \(Z_r^{p,k-p}(G)\) obtained for the
Engel group. These descriptions express each
\(Z_r^{p,k-p}(G;K)\) as the kernel of a continuous linear differential
operator of finite order on the corresponding fixed-support Fr\'echet
space. 
\end{remark}




We now define spectral currents as continuous linear functionals on the corresponding spaces of compactly supported spectral forms.

\begin{definition}[Spectral current]\label{def: currents}
The space of $(r,l)$-spectral currents of bidegree $(p,k-p)$ on $G$ is
\begin{equation*}
\mathscr C_{r,l}^{p,k-p}(G) := \bigl(\mathscr D_{l,r}^{Q-p,n - k-(Q-p)}(G)\bigr)'= \bigr\{ T \in \bigl( \mathscr D_{l,r}^{Q-p,n - k-(Q-p)}(G)\bigr)^*  \mid  T \text{ is continuous }\}.
\end{equation*}
\end{definition}

The spectral differentials on compactly supported spectral forms induce
operators on spectral currents by transposition. More precisely, suppose
that one of the spectral differentials on test forms is
\[
\Delta_r\colon
\mathscr D_{r,m}^{Q-p-r,\,
(n-k-1)-(Q-p-r)}(G)
\longrightarrow
\mathscr D_{l,r}^{Q-p,\,
(n-k)-(Q-p)}(G)\,.
\]
For
\(T\in\mathscr C_{r,l}^{p,k-p}(G)\), define
\begin{equation}\label{eq: pairing duali}
\left\langle\partial_rT,[\eta]\right\rangle
:=
(-1)^{k+1}
\left\langle T,\Delta_r[\eta]\right\rangle\ \text{ for every }[\eta]\in
\mathscr D_{r,m}^{Q-p-r,\,
(n-k-1)-(Q-p-r)}(G)\,.
\end{equation}

It follows that
\[
\partial_r\colon
\mathscr C_{r,l}^{p,k-p}(G)
\longrightarrow
\mathscr C_{m,r}^{p+r,\,
(k+1)-(p+r)}(G)\,.
\]
We note that the symbol \(\partial_r\) was already used in Section \ref{section 1} to denote the operators \eqref{eq: operators curly}. Nevertheless, we keep the same notation here for the transposed spectral differential, since its meaning will always be clear from the context.

In the Engel case, each spectral differential is induced by a
support-preserving linear differential operator of finite order.
It is therefore continuous on the corresponding fixed-support spaces
and hence on the spaces of compactly supported spectral forms. By the
definition of the quotient topology, the induced map between the
spectral quotient spaces is continuous. Consequently,
\(\partial_rT\) is again a continuous linear functional. Moreover, since consecutive spectral differentials on test forms
compose to zero, the corresponding transposed operators on spectral
currents also compose to zero.

Thus, as for the spectral complexes, the different admissible choices
of \(p\), \(r\), and \(l\) give rise to the following two complexes of
spectral currents:
\[\begin{tikzcd}
	{\big(\Omega^{7, -3}_c(G)\big)' } & {\mathscr{C}_{3,1}^{1,0}(G)} & {\mathscr{C}_{2,3}^{4,-2}(G)} & {\mathscr{C}_{1,2}^{6,-3}(G)} & {(\mathcal{C}^\infty_c(G))'} \\
	{\big(\Omega^{7, -3}_c(G)\big)' } & {\mathscr C_{2,1}^{1,0}(G)} & {\mathscr C_{3,2}^{3,-1}(G)} & {\mathscr C_{1,3}^{6,-3}(G)} & {(\mathcal{C}^\infty_c(G))'}
	\arrow["{\partial_1}", from=1-1, to=1-2]
	\arrow["{\partial_3}", from=1-2, to=1-3]
	\arrow["{\partial_2}", from=1-3, to=1-4]
	\arrow["{\partial_1}", from=1-4, to=1-5]
	\arrow["{\partial_1}", from=2-1, to=2-2]
	\arrow["{\partial_2}", from=2-2, to=2-3]
	\arrow["{\partial_3}", from=2-3, to=2-4]
	\arrow["{\partial_1}", from=2-4, to=2-5]
\end{tikzcd}\]

The spaces appearing in the two complexes above can be described explicitly.

\begin{lemma}
    Let $G$ be the Engel group. The following isomorphisms hold: 
    \begin{enumerate}
        \item\label{caso 0} $\big(\Omega^{7,-3}_c(G)\big)' \cong \mathcal{D}'(G)$,
        \item\label{caso 1} $\mathscr{C}_{3,1}^{1,0}(G) \cong \{T_1\theta_1 + T_2\theta_2 \in \mathcal{D}'(G) \otimes \bigwedge^{1,0}\mathfrak{g}^\ast \mid (X_2X_1-X_3)T_2 = X_2^2T_1\}$,
        \item\label{caso 2} $\mathscr{C}_{2,3}^{4,-2}(G) \cong \mathcal{D}'(G) \otimes \operatorname{span}_{\mathbb{R}}\{\theta_1\wedge \theta_4\} = \mathcal{D}'(G) \otimes E_{2,3}^{4,-2}(G) $,
        \item\label{caso 3} $\mathscr{C}_{1,2}^{6,-3}(G) \cong \mathcal{D}'(G) \otimes \bigwedge^{6,-3}\mathfrak{g}^\ast = \mathcal{D}'(G) \otimes E_{1,2}^{6,-3}(G)$,
        \item\label{caso 4} $\big(\mathcal{C}^\infty_c(G)\big)'\cong\mathcal{D}'(G)\otimes \bigwedge^4\mathfrak{g}^\ast$,
        \item\label{caso 5} $\mathscr{C}_{2,1}^{1,0}(G) \cong \mathcal{D}'(G) \otimes \bigwedge^{1,0}\mathfrak{g}^\ast = \mathcal{D}'(G) \otimes E_{2,1}^{1,0}(G)$,
        \item\label{caso 6} $\mathscr{C}_{3,2}^{3,-1}(G) \cong \mathcal{D}'(G) \otimes \operatorname{span}_{\mathbb{R}}\{\theta_2\wedge \theta_3\} = \mathcal{D}'(G) \otimes E_{3,2}^{3,-1}(G) $,
        \item\label{caso 7} $\mathscr{C}_{1,3}^{6,-3}(G) \cong  \left(\mathcal{D}'(G) \otimes \bigwedge^{6,-3}\mathfrak{g}^\ast\right) \big/ \left\{ T_1 \theta_1 \wedge \theta_3 \wedge \theta_4 + T_2 \theta_2 \wedge \theta_3 \wedge \theta_4 \mid X_2(T_1) = X_1(T_2)\right\}$. 
    \end{enumerate}
    Here $\mathcal{D}'(G) := \big(\mathcal{C}_c^\infty(G)\big)'$ and the vector fields $X_i$ are understood in the distributional sense.
\end{lemma}
\begin{proof}
    We first consider \eqref{caso 0} and \eqref{caso 4}. Since
    \begin{equation*}
    \dim\left(\bigwedge\nolimits^4\mathfrak g^*\right)=1, 
    \end{equation*}
    fixing the left-invariant volume form $\operatorname{vol} = \theta_1\wedge\theta_2\wedge\theta_3\wedge\theta_4$  gives the natural identification 
    \begin{equation*}
        \Omega_c^{7,-3}(G) = \mathcal C^\infty_c(G)\operatorname{vol} \cong \mathcal C^\infty_c(G).
    \end{equation*}
     Taking continuous duals, we obtain 
     \begin{equation*}
         \bigl(\Omega_c^{7,-3}(G)\bigr)' \cong \bigl(\mathcal C^\infty_c(G)\bigr)' = \mathcal D'(G),
     \end{equation*}
     which proves \eqref{caso 0} and \eqref{caso 4}.
     We next prove \eqref{caso 2}, \eqref{caso 3}, \eqref{caso 5}, and \eqref{caso 6}. The wedge product induces a non-degenerate pairing between components of complementary degree and weight. In particular,
    \begin{align*}
        \left( \operatorname{span}_{\mathbb R} \{\theta_2\wedge\theta_3\} \right)^* \cong \operatorname{span}_{\mathbb R} \{\theta_1\wedge\theta_4\}, \qquad \left( \bigwedge\nolimits^{1,0}\mathfrak g^* \right)^* \cong \bigwedge\nolimits^{6,-3}\mathfrak g^*.
    \end{align*}
    As,
    \begin{align*}
        \mathscr D_{3,2}^{3,-1}(G) \cong \mathcal C^\infty_c(G) \otimes \operatorname{span}_{\mathbb R} \{\theta_2\wedge\theta_3\}\ &\text{ and }\
        \mathscr D_{2,3}^{4,-2}(G) \cong \mathcal C^\infty_c(G) \otimes \operatorname{span}_{\mathbb R} \{\theta_1\wedge\theta_4\}\\
        \mathscr D_{2,1}^{1,0}(G) \cong \mathcal C^\infty_c(G) \otimes \bigwedge\nolimits^{1,0}\mathfrak g^*\ &\text{ and }\
        \mathscr D_{1,2}^{6,-3}(G) \cong \mathcal C^\infty_c(G) \otimes \bigwedge\nolimits^{6,-3}\mathfrak g^*,
    \end{align*}
     taking continuous duals and using the preceding identifications gives \eqref{caso 2}, \eqref{caso 3}, \eqref{caso 5}, and \eqref{caso 6}.\\
    We now prove \eqref{caso 1}. By definition,
    \begin{equation*}
        \mathscr C_{3,1}^{1,0}(G) = \left( \mathscr D_{1,3}^{6,-3}(G) \right)'\,,
    \end{equation*}
    where 
    \begin{equation*}
        \mathscr D_{1,3}^{6,-3}(G)
    =
    \left (Z_{1}^{6,-3}(G) \cap \Omega_c^3(G) 
    \right)\big/ \left(
    B_{3}^{6,-3}(G) \cap \Omega_c^3(G) \right).
    \end{equation*}
    Moreover,
    \begin{equation*}
         Z_1^{6,-3}(G)\cap\Omega_c^3(G) = \Omega_c^{6,-3}(G). 
    \end{equation*}
    Hence the continuous dual of the quotient can be naturally identified with the annihilator of $B_3^{6,-3}(G)\cap\Omega_c^3(G)$  inside $\bigl(\Omega_c^{6,-3}(G)\bigr)' \cong \mathcal D'(G)\otimes \bigwedge\nolimits^{1,0}\mathfrak g^*$. More precisely,  
    \begin{equation*}
        T = T_1 \theta_1 + T_2 \theta_2 \in \mathscr{C}_{3,1}^{1,0}(G) \Longleftrightarrow T_1(X_2^2f) - T_2(X_3f+X_1X_2f) = 0 \ , \ \forall\,\, f \in \mathcal{C}^\infty_c(G)\,.
    \end{equation*}
    Since the left-invariant vector fields $X_i$ are formally skew-adjoint with respect to Haar measure,
    \begin{equation*}
        T = T_1 \theta_1 + T_2 \theta_2 \in \mathscr{C}_{3,1}^{1,0}(G) \Longleftrightarrow  (X_2X_1-X_3)T_2 = X_2^2T_1. 
    \end{equation*}
    It remains to prove \eqref{caso 7}.
    By definition,
    \begin{equation*}
        \mathscr C_{1,3}^{6,-3}(G) = \left( \mathscr D_{3,1}^{1,0}(G) \right)',
    \end{equation*}
    and, as in this bidegree the corresponding boundary space is trivial, $\mathscr{D}_{3,1}^{1,0}(G) = Z_3^{1,0}(G)\cap\Omega_c^1(G)$.
    Consequently, by the Hahn-Banach theorem
    \begin{equation*}
        \left( Z_3^{1,0}(G)\cap\Omega_c^1(G) \right)' \cong \left(\Omega_c^{1,0}(G)\right)' \Big/ \left( Z_3^{1,0}(G)\cap\Omega_c^1(G) \right)^\perp,
    \end{equation*}
    where $\perp$ denotes the annihilator with respect to the natural
duality pairing induced by the wedge product.
    Note that 
    \begin{equation*}
        \Omega_c^{1,0}(G) \cong \mathcal C^\infty_c(G) \otimes \bigwedge\nolimits^{1,0}\mathfrak g^* \cong \mathscr D_{2,1}^{1,0}(G).
    \end{equation*}
    Therefore,
    \begin{equation*}
        \left(\Omega_c^{1,0}(G)\right)' \cong \mathscr{C}_{1,2}^{6,-3}(G) \cong \mathcal{D}'(G) \otimes \bigwedge\nolimits^{6,-3}\mathfrak{g}^\ast.
    \end{equation*}
    
    Moreover, by Lemma \ref{lemma: Z3 a supporto compatto in deg 1},
    \begin{equation*}
        Z_3^{1,0}(G) \cap \Omega^1_c(G) = \left\{ X_1 f \theta_1 + X_2 f\theta_2 \,\, \big | \,\, f \in \mathcal{C}_c^\infty(G)\right\}.
    \end{equation*}
    Therefore
    \begin{equation*}
        T=T_1\,\theta_1\wedge\theta_3\wedge\theta_4 + T_2\,\theta_2\wedge\theta_3\wedge\theta_4 \in \left( Z_3^{1,0}(G)\cap\Omega_c^1(G) \right)^\perp \Longleftrightarrow T_2(X_1f)-T_1(X_2f)=0\ , \ \forall\, f\in\mathcal C^\infty_c(G)\,.
    \end{equation*}
    Using again  $X_i^t=-X_i$,  the latter condition is equivalent to $X_2T_1=X_1T_2$.
 This concludes the proof.
\end{proof}

The explicit description obtained above shows that, in the case of the Engel group, the complexes of spectral currents have precisely the form one would expect from a distributional extension of the spectral complexes: smooth coefficients are replaced by distributions, while the differential constraints and quotient relations defining the corresponding spectral spaces persist in the distributional sense.
As in the classical theory of currents, one therefore expects smooth spectral forms to define spectral currents by integration. The following theorem shows that this is indeed the case and, moreover, that the resulting maps give rise to an embedding of complexes.

\begin{theorem}[Natural embedding into spectral currents]\label{thm: embedding spectral currents}
    Let $G$ be the Engel group. For every admissible choice of the indices $p,k,r,l$, the map 
    \begin{equation} \label{eq: linear functional embedding} 
    \psi_{r,l}^{p,k-p}:E_{r,l}^{p,k-p}(G) \longrightarrow \mathscr C_{r,l}^{p,k-p}(G), \end{equation} 
    defined by 
\begin{equation}\label{eq: definition psi spectral currents}
\psi_{r,l}^{p,k-p}([\alpha])([\eta])
:=
\int_G\alpha\wedge\eta\ ,\ \text{ for } [\eta] \in \mathscr D_{l,r}^{Q-p,\, n-k-(Q-p)}(G)\,.
\end{equation}
       is a well-defined injective linear map. Moreover, the collection of maps $\{\psi_{r,l}^{p,k-p}\}$ is compatible with the spectral differentials and the boundary operators on spectral currents. More precisely, whenever the indices are admissible, 
    \begin{equation*} 
    \psi_{r,l}^{p,k-p} \circ \Delta_l = \partial_l \circ \psi_{l,m}^{p-l,k-1-p+l}. 
    \end{equation*}   
\end{theorem}

The proof is based on the following two lemmas. We first consider the bilinear pairing underlying \eqref{eq: definition psi spectral currents}.

\begin{lemma}\label{lemma: pairing spectral currents}
    For every admissible choice of $p,k,r,l$, define
    \begin{equation*} 
    L_{r,l}^{p,k-p}: E_{r,l}^{p,k-p}(G) \times \mathscr D_{l,r}^{Q-p,\, n-k-(Q-p)}(G) \longrightarrow \mathbb R \end{equation*} 
    by 
    \begin{equation*} L_{r,l}^{p,k-p}([\alpha],[\eta]) := \int_G\alpha\wedge\eta. 
    \end{equation*} 
    Then the following properties hold:
\begin{enumerate}
    \item\label{pairing well defined} the pairing $L_{r,l}^{p,k-p}$ is well-defined; \item\label{pairing non degenerate} the pairing $L_{r,l}^{p,k-p}$ is non-degenerate;
    \item\label{pairing Stokes} for every $[\nu] \in E_{l,m}^{p-l,k-1-p+l}(G)$ and $[\eta] \in \mathscr{D}_{l,r}^{Q-p, n- k - Q + p }(G)$,
    \begin{equation*}
        L_{r,l}^{p,k-p}([\Delta_l \nu], [\eta]) = (-1)^{k} L_{l,m}^{p-l,k-1-p+l}( [\nu], [\Delta_l \eta])
    \end{equation*}
\end{enumerate}
\end{lemma}

\begin{proof}
    We first prove \eqref{pairing well defined}. Let $\alpha \in Z_r^{p,k-p}(G)$ and $\beta \in B_r^{Q-p,n-k-Q+p}(G)$, and assume that at least one of them has compact support. 
  By Definition \ref{def: Z and B defined}, there exist forms $z_{p+j}\in\Omega^{p+j,k-p-j}(G)$ for $j = 1,\ldots, r-1$ such that 
    \begin{align*}
        d\big( \alpha - \sum_{j=1}^{r-1}z_{p+j} \big) \in \Omega^{\geq p+r, k+1}(G).
    \end{align*}    
    The same definition also yields forms $\zeta_{Q-p-i} \in \Omega^{Q-p -i, n-(k+1) - Q + p +i}(G)$ for $i=0,\ldots,r-1$ such that
    \begin{align*}
        d\big( \sum_{i=0}^{r-1}\zeta_{Q-p-i}\big) = \beta  + \Omega^{Q-p+1, n - k}(G)\,.
    \end{align*}
Note that, for every $j=1,\ldots,r-1$, we have $w(z_{p+j}\wedge\beta)\geq Q+1,$ and hence $z_{p+j}\wedge\beta=0$. Moreover, the terms in
$d\left(\sum_{i=0}^{r-1}\zeta_{Q-p-i}\right)-\beta$ have weight at least $(Q-p+1)$, and therefore their wedge product with
$\alpha-\sum_{j=1}^{r-1}z_{p+j}$, whose terms have weight at least $p$, vanishes. Consequently,
\begin{align*}
        \int_G \alpha \wedge \beta &= \int_G \big( \alpha - \sum_{j=1}^{r-1}z_{p+j} \big) \wedge \beta= \int_G \big( \alpha - \sum_{j=1}^{r-1}z_{p+j} \big) \wedge d\big( \sum_{i=0}^{r-1}\zeta_{Q-p-i}\big)\\
        &= (-1)^{k+1} \int_G d\big( \alpha - \sum_{j=1}^{r-1}z_{p+j} \big) \wedge \big( \sum_{i=0}^{r-1}\zeta_{Q-p-i}\big) = 0.
    \end{align*}
    Therefore, for every $\beta_{\alpha} \in B_l^{p,k-p}$ and $\beta_{\eta} \in B_r^{Q-p,n-k-Q+p}(G)$ 
    \begin{equation*}
        \int_G \big(\alpha + \beta_{\alpha} \big) \wedge \big(\eta + \beta_{\eta} \big) = \int_G \alpha\wedge \eta,
    \end{equation*}
    i.e. the integral does not depend on either representative.
    This proves that the map $L_{r,l}^{p,k-p}$ is well defined.

    We now prove \eqref{pairing non degenerate}. For all the spectral spaces whose explicit descriptions contain no differential constraint, the claim follows immediately from the non-degeneracy of the wedge product between complementary components.
    For the dual pair 
    \begin{equation*}
    E_{3,1}^{1,0}(G) \times \mathscr D_{1,3}^{6,-3}(G),
    \end{equation*} 
    let $\alpha = f_1\theta_1+f_2\theta_2 \in Z_3^{1,0}(G)$ be non-zero. Assume, without loss of generality, that $f_1\not\equiv0$. Choose $x_0\in G$ such that $f_1(x_0)\neq0$, and let $\chi\in\mathcal C_c^\infty(G)$ be a non-negative cut-off function supported in a sufficiently small neighbourhood of $x_0$ on which $f_1$ does not vanish identically. Set 
    \begin{equation*} 
    \eta = \chi f_1\, \theta_2\wedge\theta_3\wedge\theta_4. \end{equation*}
    Then 
    \begin{equation*} 
    \int_G\alpha\wedge\eta = \int_G \chi f_1^2\,\operatorname{vol} >0. \end{equation*} 
    Thus no non-zero element of $E_{3,1}^{1,0}(G)$ annihilates the whole space $\mathscr D_{1,3}^{6,-3}(G)$.
    Conversely, 
    for 
    \begin{equation*}
        E_{1,3}^{6,-3}(G) \times \mathcal{D}_{3,1}^{1,0}
    \end{equation*}
    let $\eta = h_1\, \theta_1\wedge\theta_3\wedge\theta_4 + h_2\, \theta_2\wedge\theta_3\wedge\theta_4$ represent a non-zero class in $E_{1,3}^{6,-3}(G)$. Thus, $\eta \notin B_3^{6,-3}(G)$. By Lemma \ref{lemma: B esplicito nel caso liscio in grado 3}, this is equivalent to 
    \begin{equation*} 
    X_2h_1-X_1h_2\not\equiv0. 
    \end{equation*}
    On the other hand, by Lemma \ref{lemma: Z3 a supporto compatto in deg 1}, every compactly supported form in $Z_{3}^{1,0}(G) \cap \Omega^1_c(G)$ can be written as
    \begin{equation*} X_1f\,\theta_1+X_2f\,\theta_2, \qquad f\in\mathcal C_c^\infty(G), \end{equation*}
    Therefore, 
    \begin{align*} 
    \int_G \left( X_1f\,\theta_1+X_2f\,\theta_2 \right) \wedge\eta &= \int_G \left( h_2X_1f-h_1X_2f \right) \operatorname{vol} = \int_G f\left( X_2h_1-X_1h_2 \right) \operatorname{vol}. 
    \end{align*} 
    Since $X_2h_1-X_1h_2\not\equiv0$, arguing as above, one can choose $f\in\mathcal C_c^\infty(G)$ so that the last integral is non-zero. This proves non-degeneracy.
    
    Finally, we prove \eqref{pairing Stokes}.
   Choose representatives $\nu \in [\nu]$ and $\eta \in [\eta]$. By the description of the spectral differentials in Proposition \ref{prop: differentials Delta_r}, there exist  $ z_{p-l+j}\in \Omega^{p-l+j, k-1 -p+l-j}(G)$ for $1\le j\le l-1$ such that
    \begin{align*}
        L_{r,l}^{p,k-p}([\Delta_l \nu], [\eta]) &= \int_G \Delta_l\nu \wedge \eta = \int_G  \big(d_l\nu-\sum_{j=1}^{l-1}d_jz_{p-j}\big) \wedge \eta = \int_G d \big(\nu-\sum_{j=1}^{l-1}z_{p-j}\big) \wedge \eta. 
    \end{align*}
    The last equality follows from the defining relations satisfied by the family
    $\lbrace z_{p-l+j}\rbrace_{1\le j\le l-1}$; see Definition \ref{def: Z and B defined}.
    Similarly, since $\eta\in \mathscr{D}_{l,r}^{Q-p,n-k-Q+p}(G)$, there exist forms $\zeta_{Q-p+j}\in \Omega^{Q-p+j, n-k -Q+p-j}(G)$ for $1\le j\le l-1$ such that
    \begin{align*}
        L_{l,m}^{p-l,k-1-p+l}([\nu], [\Delta_l \eta]) &= \int_G \nu \wedge \Delta_l \eta= \int_G \nu \wedge \big( d_l \eta - \sum_{j=1}^{l-1}d_j\zeta_{Q-p+l-j} \big) = \int_G \nu \wedge d\big(\eta - \sum_{j=1}^{l-1}\zeta_{Q-p+l-j}\big) .
    \end{align*}
    Moreover, the defining relations for the families $\lbrace z_{p-l+j}\rbrace_{1\le j\le l-1}$ and $\lbrace \zeta_{Q-p+j}\rbrace_{1\le j\le l-1}$ imply that, for every $1\le j\le l-1$,
    \begin{align*}
        w\Big(z_{p-j} \wedge d\big(\eta - \sum_{j=1}^{l-1}\zeta_{Q-p+l-j}\big)\Big) &\geq Q +1\,, \\
        w\Big(d \big(\nu-\sum_{j=1}^{l-1}z_{p-j}\big) \wedge \zeta_{Q-p+j} \Big) &\geq Q+1\,.
    \end{align*}
    Hence all these wedge products vanish, and
    \begin{align*}
       L_{r,l}^{p,k-p}([\Delta_l \nu], [\eta]) &= \int_G d \big(\nu-\sum_{j=1}^{l-1}z_{p-j}\big) \wedge \big(\eta - \sum_{j=1}^{l-1}\zeta_{Q-p+l-j}\big)\,,\\
       L_{l,m}^{p-l,k-1-p+l}([\nu], [\Delta_l \eta]) &=\int_G \big(\nu-\sum_{j=1}^{l-1}z_{p-j}\big) \wedge d\big(\eta - \sum_{j=1}^{l-1}\zeta_{Q-p+l-j}\big)\,.
    \end{align*}
   The identity \eqref{pairing Stokes} now follows directly from Stokes' theorem.     
\end{proof}

\begin{remark}
    Points \eqref{pairing well defined} and \eqref{pairing Stokes} hold in the more general setting of an arbitrary Carnot group $G$. In contrast, the proof of point \eqref{pairing non degenerate} relies on the explicit structure of the spaces $Z_r^{p,k-p}(G)$ and $B_r^{p,k-p}(G)$ in the Engel group.
\end{remark}

We next prove the continuity of the functional induced by $\alpha \in E_{r,l}^{p,k-p}(G)$

\begin{lemma}\label{lemma: continuity spectral embedding}
    Let $[\alpha]\in E_{r,l}^{p,k-p}(G)$. Then the linear functional on $\mathscr D_{l,r}^{Q-p,\, n-k-(Q-p)}(G)$
    \begin{equation*}
    \psi_{r,l}^{p,k-p}([\alpha]) = L_{r,l}^{p,k-p}([\alpha], \cdot) \colon[\eta] \longmapsto \int_G\alpha\wedge\eta \end{equation*}
    is continuous. 
\end{lemma}
\begin{proof}
    Choose a smooth representative $\alpha\in Z_r^{p,k-p}(G)$. Let $ F_\alpha $ be the functional on $ Z_l^{Q-p,\, n-k-(Q-p)}(G) \cap \Omega_c^{n-k}(G)$ defined as
    \begin{equation*} 
    F_\alpha(\eta) := \int_G\alpha\wedge\eta \end{equation*}   
    Let $K\subset G$ be compact and $\eta \in Z_l^{Q-p,\, n-k-(Q-p)}(G;K)$.
    Since $\alpha$ is smooth, there exists a constant $C_{\alpha,K}>0$ such that \begin{equation*} 
     \left| F_\alpha(\eta) \right| =\left| \int_K\alpha\wedge\eta \right| \leq \int_K |(\alpha \wedge \eta)(x) |dVol(x) \leq C_{\alpha,K} \mathfrak p_{K,0}(\eta). 
     \end{equation*}
     Hence the map $F_\alpha$ is continuous on every Fr\'echet space $Z_l^{Q-p,\, n-k-(Q-p)}(G;K)$. 
     By the universal property of the inductive-limit topology, $F_\alpha$ is therefore continuous on $Z_l^{Q-p,\, n-k-(Q-p)}(G) \cap \Omega_c^{n-k}(G)$ endowed with its LF topology. Since $F_\alpha$ vanishes on $B_r^{Q-p,\, n-k-(Q-p)}(G) \cap \Omega_c^{n-k}(G), $ the universal property of the quotient topology implies that it descends to a continuous linear functional on $\mathscr D_{l,r}^{Q-p,\, n-k-(Q-p)}(G)$.
\end{proof}

\begin{proof}[Proof of Theorem \ref{thm: embedding spectral currents}] 
By point \eqref{pairing well defined} of Lemma \ref{lemma: pairing spectral currents}, the expression in \eqref{eq: definition psi spectral currents} is well defined, and by Lemma \ref{lemma: continuity spectral embedding} it defines an element of $\mathscr C_{r,l}^{p,k-p}(G)$. Moreover, point \eqref{pairing non degenerate} of Lemma \ref{lemma: pairing spectral currents} implies that the map $\psi_{r,l}^{p,k-p}$ is injective, while point \eqref{pairing Stokes} gives the compatibility with the differentials. Hence, the family $\{\psi_{r,l}^{p,k-p}\}$ defines an embedding of complexes.
\end{proof}

\begin{remark}
    Theorem \ref{thm: embedding spectral currents} applies separately to each complex in the family of spectral complexes associated with the Engel group, and therefore yields two embeddings. In other words, the theorem states that
\[\begin{tikzcd}
	{Z_1^{0,0}(G)} & {E_{2,1}^{1,0}(G)} & {E_{3,2}^{3,-1}(G)} & {E_{1,3}^{6,-3}(G)} & {\Omega^{7,-3}(G)} \\
	{\mathcal{D}'(G)} & {\mathscr{C}_{2,1}^{1,0}(G)} & {\mathscr{C}_{3,2}^{3,-1}(G)} & {\mathscr{C}_{1,3}^{6,-3}(G)} & {\mathcal{D}'(G) \otimes \bigwedge^4\mathfrak{g}^\ast}
	\arrow["{\Delta_1}", from=1-1, to=1-2]
	\arrow[hook, from=1-1, to=2-1]
	\arrow["{\Delta_2}", from=1-2, to=1-3]
	\arrow[hook, from=1-2, to=2-2]
	\arrow["{\Delta_3}", from=1-3, to=1-4]
	\arrow[hook, from=1-3, to=2-3]
	\arrow["{\Delta_1}", from=1-4, to=1-5]
	\arrow[hook, from=1-4, to=2-4]
	\arrow[hook, from=1-5, to=2-5]
	\arrow["{\partial_1}"', from=2-1, to=2-2]
	\arrow["{\partial_2}"', from=2-2, to=2-3]
	\arrow["{\partial_3}"', from=2-3, to=2-4]
	\arrow["{\partial_1}"', from=2-4, to=2-5]
\end{tikzcd}\]

and

\[\begin{tikzcd}
	{Z_1^{0,0}(G)} & {E_{3,1}^{1,0}(G)} & {E_{2,3}^{4,-2}(G)} & {E_{1,2}^{6,-3}(G)} & {\Omega^{7,-3}(G)} \\
	{\mathcal{D}'(G)} & {\mathscr{C}_{3,1}^{1,0}(G)} & {\mathscr{C}_{2,3}^{4,-2}(G)} & {\mathscr{C}_{1,2}^{6,-3}(G)} & {\mathcal{D}'(G) \otimes \bigwedge^4\mathfrak{g}^\ast}
	\arrow["{\Delta_1}", from=1-1, to=1-2]
	\arrow[hook, from=1-1, to=2-1]
	\arrow["{\Delta_3}", from=1-2, to=1-3]
	\arrow[hook, from=1-2, to=2-2]
	\arrow["{\Delta_2}", from=1-3, to=1-4]
	\arrow[hook, from=1-3, to=2-3]
	\arrow["{\Delta_1}", from=1-4, to=1-5]
	\arrow[hook, from=1-4, to=2-4]
	\arrow[hook, from=1-5, to=2-5]
	\arrow["{\partial_1}"', from=2-1, to=2-2]
	\arrow["{\partial_3}"', from=2-2, to=2-3]
	\arrow["{\partial_2}"', from=2-3, to=2-4]
	\arrow["{\partial_1}"', from=2-4, to=2-5]
\end{tikzcd}\]

\end{remark}

\section{Pansu pullback as a morphism into spectral currents}

Let $G$ be a Carnot group, and let $n$ and $Q$ denote its topological and homogeneous dimensions, respectively. Let $\varphi:G \rightarrow G$  belong to $W^{1,q}_{\mathrm{loc}}(G, G)$ with $q> Q$. By Theorem \ref{thm: Vodopyanov Pansu diff}, the map $\varphi$ is Pansu differentiable almost every $x \in G$.

\begin{definition}
    Let $\alpha=f\Theta\in\Omega^k(G)$ where $f\in\mathcal C^\infty(G)$ and $\Theta\in\bigwedge^k\mathfrak g^*$ is left-invariant. We define the Pansu pullback of $\alpha$ by $\varphi$ at almost every $x\in G$ by \begin{equation}\label{eq: Pansu pullback Sobolev}
    \varphi_P^*\alpha(x) := (f\circ\varphi)(x)\, D\varphi(x)^*\Theta, 
    \end{equation}
    where $D\varphi(x)= \exp^{-1}_{ G} \circ D_P\varphi(x)\circ \exp_G\colon \mathfrak g \longrightarrow \mathfrak g$ is the Lie algebra homomorphism induced by the Pansu derivative in $x$.
    Equivalently, for every $Y_1,\ldots,Y_k\in\mathfrak g$, \begin{equation*}
  \bigl(\varphi_P^*\alpha(x)\bigr) (Y_1,\ldots,Y_k) = (f\circ\varphi)(x) \Theta \bigl( D\varphi(x)Y_1,\ldots,D\varphi(x)Y_k \bigr). 
    \end{equation*}
    The definition extends to arbitrary smooth differential forms by linearity.
\end{definition}

\begin{lemma}\label{lemma: coeff del pullback L1}
Let $\Theta\in\bigwedge^{p,k-p}\mathfrak g^*$. Then, for $k \ge 1$
\begin{equation*}
\varphi_P^*(f\Theta) \in L^{q/p}_{\mathrm{loc}} \left( G,\bigwedge\nolimits^k\mathfrak g^* \right). \end{equation*}
In particular, for every $\alpha\in\Omega^k(G)$, 
\begin{equation*} 
\varphi_P^*\alpha \in L^1_{\mathrm{loc}} \left( G,\bigwedge\nolimits^k\mathfrak g^* \right). 
\end{equation*} 
For $k=0$, one has 
\begin{equation*} 
\varphi_P^*f=f\circ\varphi \in L^\infty_{\mathrm{loc}}(G).
\end{equation*}
\end{lemma}
\begin{proof} 
Let $K\subset G$ be compact. Since $q>Q$, the Sobolev-Morrey embedding (see \cite[Lemma 2.7]{kleiner2020pansu}) gives a continuous representative of $\varphi$. In particular, $\varphi(K)$ is compact, and hence 
\begin{equation*} 
\sup_{x\in K}|f(\varphi(x))|<\infty. \end{equation*} 
Let $D_h\varphi(x):V_1\longrightarrow V_1$ denote the formal horizontal differential.
Since $\varphi\in W^{1,q}_{\mathrm{loc}}(G,G)$, 
\begin{equation*} 
|D_h\varphi| \in L^q_{\mathrm{loc}}(G).
\end{equation*}
Moreover, as $D\varphi(x)$ is a graded Lie algebra homomorphism, its restriction to every higher layer is determined by $D_h\varphi(x)$, for every $j$ there exists a constant $C_j>0$, depending only on the group, such that 
\begin{equation*} 
\left| D\varphi(x)|_{V_j} \right| \leq C_j |D_h\varphi(x)|^j 
\end{equation*} 
for almost every $x\in G$. Therefore, there exists a constant $C_\Theta>0$ such that 
\begin{equation*} 
\left| D\varphi(x)^*\Theta \right| \leq C_\Theta |D_h\varphi(x)|^p 
\end{equation*}
for almost every $x\in G$. It follows that \begin{equation*} 
|\varphi_P^*(f\Theta)(x)| \leq C_\Theta |f(\varphi(x))| |D_h\varphi(x)|^p.
\end{equation*} 
Since $f\circ\varphi$ is bounded on $K$ and $D_h\varphi\in L^q(K)$, we obtain \begin{equation*} 
\varphi_P^*(f\Theta) \in L^{q/p} \left( K,\bigwedge\nolimits^k\mathfrak g^* \right).
\end{equation*}
Finally, since $1\le p\le Q<q$, we have $\frac{q}{p}>1$ for every possible weight $p$ and, therefore, $L^{q/p}(K)\subset L^1(K)$.
\end{proof}

By Lemma \ref{lemma: coeff del pullback L1}, the Pansu pullback  of every smooth $k$-form naturally defines a classical current of degree $k$. More precisely, we have a linear map
\begin{align*}
    \Omega^k(G) &\longrightarrow \big(\Omega^{n -k}_c(G)\big)'\\
    \alpha &\mapsto \Big( \eta \mapsto \int_G \varphi_p^* \alpha \wedge \eta\Big),
\end{align*}
where $\bigl(\Omega_c^\bullet(G)\bigr)'$ denotes the classical space of currents on $G$. Equivalently, $\varphi_P^*\alpha$ can be regarded as a differential form with distributional coefficients, namely
\begin{equation*}
    \varphi_P^* \alpha \in \mathcal{D}'(G) \otimes \bigwedge\nolimits^k \mathfrak{g}^\ast\,.
\end{equation*}

Moreover, the graded structure of $D\varphi(x)$ implies that the Pansu pullback preserves the weight decomposition. More precisely, for every $p$,
\begin{equation*}
D\varphi(x)^* \left( \bigwedge\nolimits^{p,k-p}\mathfrak g^* \right) \subset \bigwedge\nolimits^{p,k-p}\mathfrak g^*
\end{equation*}
for almost every $x\in G$. Consequently, 
\begin{equation}\label{eq: Pansu pullback preserves weight} 
\alpha\in\Omega^{p,k-p}(G) \quad\Longrightarrow\quad \varphi_P^*\alpha \in L^1_{\mathrm{loc}}(G) \otimes \bigwedge\nolimits^{p,k-p}\mathfrak g^*. 
\end{equation} 
In particular, through the integration pairing,
\begin{equation*} 
\varphi_P^*\alpha \in \left( \Omega_c^{Q-p,\, n-k-(Q-p)}(G) \right)'.
\end{equation*}

We next consider the algebraic differential $d_0$. Since $d_0$ acts only on the left-invariant component of a differential form, it extends naturally to forms with distributional coefficients. Thus \begin{equation*} d_0: \mathcal D'(G)\otimes \bigwedge\nolimits^k\mathfrak g^* \longrightarrow \mathcal D'(G)\otimes \bigwedge\nolimits^{k+1}\mathfrak g^* \end{equation*} is a well-defined linear operator. The Pansu pullback commutes with this algebraic differential.
\begin{lemma}\label{lemma: commutativita d0 e pullback}
For every $\alpha \in \Omega^k(G)$, it holds
    \begin{equation*}
        d_0 \circ \varphi_P^* (\alpha) = \varphi_P^* \circ d_0 (\alpha).
    \end{equation*}
\end{lemma}
\begin{proof}
    At every point $x\in G$ at which $\varphi$ is Pansu differentiable, the induced Pansu differential
    \begin{equation*}
        D\varphi(x):\mathfrak g\longrightarrow \mathfrak g
    \end{equation*}
    is a Lie algebra homomorphism and therefore preserves Lie brackets.
Let $\alpha \in \Omega^k(G)$ be of the form $\alpha = f \Theta$ with $f \in C^{\infty}(G)$ and $\Theta \in \bigwedge\nolimits^k\mathfrak{g}^\ast$. Then, by the expression \eqref{formula esplicita d0}, for every $Y_1,\ldots,Y_{k+1}\in\mathfrak{g}$,
    \begin{align*}
    &(\varphi_P^\ast d_0\alpha)(Y_1,\ldots,Y_{k+1})=(f\circ\varphi)\,d\Theta\big((D\varphi)Y_1,\ldots,(D\varphi)Y_{k+1}\big)\\&=\sum_{1\le i<j\le k+1}(-1)^{i+j}(f\circ \varphi)\,\Theta\big([(D\varphi )Y_i,(D\varphi)Y_j],(D\varphi)Y_1,\ldots,\widehat{(D\varphi)Y_i},\ldots,\widehat{(D\varphi)Y_j},\ldots,(D\varphi)Y_{k+1}\big)\\&=\sum_{1\le i<j\le k+1}(-1)^{i+j}(f\circ\varphi)\,\Theta\big((D\varphi)[Y_i,Y_j],(D\varphi)Y_1,\ldots,\widehat{(D\varphi)Y_i},\ldots,\widehat{(D\varphi)Y_j},\ldots,(D\varphi)Y_{k+1}\big)\\&=\sum_{1\le i<j\le k+1}(-1)^{i+j}\varphi_P^\ast\alpha([Y_i,Y_j],Y_1,\ldots,\widehat{Y_i},\ldots,\widehat{Y_j},\ldots,Y_{k+1})=(d_0\varphi_P^\ast\alpha)(Y_1,\ldots,Y_{k+1})\,.
\end{align*}
The general case follows by linearity.
\end{proof}

\begin{remark}\label{rmk: Pullback e Rumin}
    A  direct consequence of Lemma \ref{lemma: commutativita d0 e pullback} is that
\begin{align*}
    \alpha \in C^\infty(G) \otimes \ker d_0 &\Longrightarrow \varphi_P^* \alpha \in \mathcal{D}'(G) \otimes \ker d_0\,,\\
    \alpha \in C^\infty(G) \otimes \operatorname{Im} d_0 &\Longrightarrow \varphi_P^* \alpha \in \mathcal{D}'(G) \otimes \operatorname{Im} d_0\,.
\end{align*}
Consequently, the Pansu pullback induces a well-defined map 
\begin{align*}
    E_0^k(G)  &\longrightarrow  \big(E_0^{n-k}(G)\cap\Omega_c^{n-k}(G)\big)'\\ \alpha &\mapsto \Big( \eta \mapsto \int_G \varphi_p^* \alpha \wedge \eta\Big)\,,
\end{align*}
where $E_0^\bullet(G)\cap\Omega_c^\bullet(G)$ denotes the space of compactly supported Rumin forms and $\big(E_0^\bullet(G)\cap\Omega_c^\bullet(G)\big)'$ is its continuous dual, i.e the space of Rumin currents \cite{FSSC6,Can21jga,JuliaPansu2023}.
\end{remark}

From now on, we specialise our considerations to the Engel group. Thus $n=4$ and $Q=7$. We show that an analogous result holds for the associated spectral complexes.

We will use the following ``integration-by-parts" formula for the Pansu pullback due to Kleiner-M\"uller-Xie, see \cite[Theorem 4.2]{kleiner2020pansu}. We state it in our specific setting and using the same convention used in \cite{biundo2026pansupullbackspectralcomplexes} where the weights of covectors are positive.

\begin{theorem}\label{th: KMX 1}
     Let $\varphi : G \rightarrow G$ be a $W_{\text{loc}}^{1,q}$-map for some $q>Q$. Let $\alpha \in \Omega^k(G)$  $\eta \in \Omega_c^{n - k -1}(G)$, where $n$ is the topological dimension of $G$. Assume that
    \begin{equation*}
       \min \left\{ w(\alpha) + w(d\eta), w(d\alpha) + w(\eta) \right\} \geq Q.
    \end{equation*}
    Then the following identity holds:
    \begin{equation*}
        \int_{G} \varphi^*_P d\alpha \wedge \eta + (-1)^{k} \int_{G}\varphi^*_P\alpha \wedge d\eta = 0.
    \end{equation*}
\end{theorem}

We first use Theorem \ref{th: KMX 1} to show that integration against the Pansu pullback is compatible with the quotient structures of the spectral spaces. 

\begin{proposition}\label{prop: well defined Pansu spectral pairing}
Let $G$ be the Engel group, and let $(p,k,r,l)$ be an admissible
quadruple occurring in one of its two spectral complexes. Let $\alpha \in Z_r^{p,k-p}(G)$, $\beta \in B_l^{p,k-p}(G)$  and $\eta \in Z_l^{Q-p,\, n-k-(Q-p)}(G) \cap \Omega_c^{n-k}(G)$. Then 

\begin{equation}\label{eq: source representative well defined} 
\int_G \varphi_P^*(\alpha+\beta)\wedge\eta = \int_G \varphi_P^*\alpha\wedge\eta. 
\end{equation}
Moreover, if $\gamma \in B_r^{Q-p,\, n-k-(Q-p)}(G) \cap \Omega_c^{n-k}(G)$, then \begin{equation}\label{eq: test representative well defined}
\int_G \varphi_P^*\alpha\wedge\gamma = 0.
\end{equation} 
\end{proposition}
\begin{proof}
We first prove \eqref{eq: source representative well defined}.
By the explicit descriptions of the spaces occurring in the two Engel
spectral complexes, the only nontrivial cases correspond to the pairs
\[
B_2^{3,-1}(G)\times
\bigl(Z_2^{4,-2}(G)\cap\Omega_c^2(G)\bigr)\ \text{ 
and
 }\
B_3^{6,-3}(G)\times
\bigl(Z_3^{1,0}(G)\cap\Omega_c^1(G)\bigr).
\]
It is enough to prove that
\[
\int_G\varphi_P^*\beta\wedge\eta=0\,.
\]

Suppose first that $\beta\in B_2^{3,-1}(G)$ and $\eta\in Z_2^{4,-2}(G)\cap\Omega_c^2(G)$.
Recall that $B_2^{3,-1}(G)
=
\mathcal C^\infty(G)\otimes
\operatorname{span}_{\mathbb R}\{\theta_1\wedge\theta_3\}$, while $Z_2^{4,-2}(G)
=
\mathcal C^\infty(G)\otimes
\operatorname{span}_{\mathbb R}\{\theta_1\wedge\theta_4\}$.
By Remark \ref{rmk: Pullback e Rumin}, $\varphi_P^*\beta
\in
L_{\mathrm{loc}}^1(G)\otimes
\operatorname{span}_{\mathbb R}\{\theta_1\wedge\theta_3\}$ and so
\[
\varphi_P^*\beta\wedge\eta=0,
\]
and the claim follows.

Let us now consider $\beta\in B_3^{6,-3}(G)$ and $\eta\in Z_3^{1,0}(G)\cap\Omega_c^1(G)$.
By the definition of \(B_3^{6,-3}(G)\), there exist $c_4\in\Omega^{4,-2}(G)$ and $c_5\in\Omega^{5,-3}(G)$
such that
\[
d(c_4+c_5)-\beta\in\Omega^{\geq7,3}(G)\,.
\]
Since the largest possible weight of a nonzero \(3\)-form on the Engel
group is \(6\), we have
\(\Omega^{\geq7,3}(G)=\{0\}\). Hence $d(c_4+c_5)=\beta$, with $w(c_4+c_5)\ge 4$ and $w(d(c_4+c_5))\ge 6$.

Since
\(\eta\in Z_3^{1,0}(G)\cap\Omega_c^1(G)\), the compact-support
convention provides forms $z_2\in\Omega_c^{2,-1}(G)$ and $z_3\in\Omega_c^{3,-2}(G)$
such that, setting $\widehat\eta:=\eta-z_2-z_3$, one has $w(\widehat\eta)\geq1$ and $w(d\widehat\eta)\geq4$.
Moreover, \(\widehat\eta\) is compactly supported.

Since the Pansu pullback preserves homogeneous weights and
\(w(\beta)=6\), we have that $\varphi_P^*\beta\wedge z_2
=
\varphi_P^*\beta\wedge z_3
=
0$, since there are no forms of weight greater than 7.
Therefore,
\[
\int_G\varphi_P^*\beta\wedge\eta
=
\int_G\varphi_P^*(d(c_4+c_5))\wedge\widehat\eta\,.
\]
Moreover, the hypotheses of Theorem \ref{th: KMX 1} are satisfied because
\[
w(c_4+c_5)+w(d\widehat\eta)\geq4+4=8\ \text{ and }\ 
w(d(c_4+c_5))+w(\widehat\eta)\geq6+1=7\,.
\]
Since \(c_4+c_5\) has degree \(2\), the same theorem gives
\[
\int_G\varphi_P^*(d(c_4+c_5))\wedge\widehat\eta
=
-\int_G\varphi_P^*{(c_4+c_5)}\wedge d\widehat\eta\,.
\]
The integrand on the right is a \(4\)-form of weight at least
\(4+4=8>Q\), and hence vanishes identically. Thus
\[
\int_G\varphi_P^*\beta\wedge\eta=0\,,
\]
which proves \eqref{eq: source representative well defined}.

Let us now prove \eqref{eq: test representative well defined}. Again,
there are only two nontrivial cases. Suppose first that $\alpha\in Z_2^{4,-2}(G)$ and  $\gamma\in B_2^{3,-1}(G)\cap\Omega_c^2(G)$. Similarly to what was done before, since $Z_2^{4,-2}(G)
=
\mathcal C^\infty(G)\otimes
\operatorname{span}_{\mathbb R}\{\theta_1\wedge\theta_4\}$,
Remark \ref{rmk: Pullback e Rumin} implies that $\varphi_P^*\alpha
\in
L_{\mathrm{loc}}^1(G)\otimes
\operatorname{span}_{\mathbb R}\{\theta_1\wedge\theta_4\}$.
On the other hand, $\gamma\in
\mathcal C_c^\infty(G)\otimes
\operatorname{span}_{\mathbb R}\{\theta_1\wedge\theta_3\}$, and so $\varphi_P^*\alpha\wedge\gamma=0$.

Finally, suppose that $\alpha\in Z_3^{1,0}(G)$ and $\gamma\in B_3^{6,-3}(G)\cap\Omega_c^3(G)$.
Since \(\alpha\in Z_3^{1,0}(G)\), there exist $z_2\in\Omega^{2,-1}(G)$ and $z_3\in\Omega^{3,-2}(G)$
such that, setting $\widehat\alpha:=\alpha-z_2-z_3$,
one has $w(\widehat\alpha)\geq1$ and $w(d\widehat\alpha)\geq4$.

By the compact-support definition of
\(B_3^{6,-3}(G)\cap\Omega_c^3(G)\), there exist compactly supported
forms $c_4\in\Omega_c^{4,-2}(G)$ and $c_5\in\Omega_c^{5,-3}(G)$
such that $d(c_4+c_5)-\gamma\in\Omega_c^{\geq7,3}(G)$.
As before, \(\Omega^{\geq7,3}(G)=\{0\}\), and therefore $d(c_4+c_5)=\gamma$, $w(c_4+c_5)\ge 4$, and $w(d(c_4+c_5))\ge 6$.

Since \(w(\gamma)=6\), the usual weight considerations give $\varphi_P^*z_2\wedge\gamma
=
\varphi_P^*z_3\wedge\gamma
=
0$, and so it follows that
\[
\int_G\varphi_P^*\alpha\wedge\gamma
=
\int_G\varphi_P^*\widehat\alpha\wedge d(c_4+c_5)\,.
\]
The hypotheses of Theorem \ref{th: KMX 1} are again satisfied because
\[
w(\widehat\alpha)+w(d(c_4+c_5))\geq1+6=7\ \text{ and }\
w(d\widehat\alpha)+w(c_4+c_5)\geq4+4=8\,.
\]
Applying the theorem to the \(1\)-form \(\widehat\alpha\) and the
compactly supported \(2\)-form \(c_4+c_5\), we obtain
\[
\int_G\varphi_P^*\widehat\alpha\wedge d(c_4+c_5)
=
\int_G\varphi_P^*(d\widehat\alpha)\wedge (c_4+c_5)\,.
\]
The integrand on the right is a \(4\)-form of weight at least
\(4+4=8>Q\), so it vanishes identically. Consequently,
\[
\int_G\varphi_P^*\alpha\wedge\gamma=0.
\]
This proves \eqref{eq: test representative well defined} and concludes
the proof.
\end{proof}

The proposition immediately yields the following. 

\begin{corollary} \label{cor: buona def pullback come integrale sulle spectral complexes} 
For every $[\alpha]\in E_{r,l}^{p,k-p}(G)$, the functional on $Z_l^{Q-p, n-k-(Q-p)}(G) \cap \Omega_c^{n-k}(G)$ defined by
\begin{align*}
    \eta \longmapsto \int_G \varphi_P^*\alpha\wedge\eta
\end{align*}
depends only on the class $[\alpha]$ and vanishes on
\begin{equation*}
B_r^{Q-p,\, n-k-(Q-p)}(G) \cap \Omega_c^{n-k}(G). 
\end{equation*}
Consequently, it descends to a well-defined linear functional 
\begin{equation}\label{eq: def del pullback come corrente}
\phi_{[\alpha]}: \mathscr D_{l,r}^{Q-p,\, n-k-(Q-p)}(G) \longrightarrow \mathbb R. 
\end{equation} 
\end{corollary}

We next verify that this functional is continuous. 
\begin{lemma} \label{lemma: continuitá pullback come corrente}
For every $[\alpha]\in E_{r,l}^{p,k-p}(G),$ the functional $\phi_{[\alpha]}$ is continuous. In particular, 
\begin{equation*} 
\phi_{[\alpha]} \in \mathscr C_{r,l}^{p,k-p}(G). 
\end{equation*} 
\end{lemma}
\begin{proof} 
Let $K\subset G$ be compact and let $\eta \in Z_l^{Q-p,\, n-k-(Q-p)}(G;K)$.
Since 
\begin{equation*}
\varphi_P^*\alpha \in L^1_{\mathrm{loc}} \left( G,\bigwedge\nolimits^k\mathfrak g^* \right), 
\end{equation*}
there exists a constant $C>0$, depending only on the fixed left-invariant norm on forms, such that 
\begin{equation*}
\left| \int_G \varphi_P^*\alpha\wedge\eta \right| \leq C \left\| \varphi_P^*\alpha \right\|_{L^1(K)} \mathfrak p_{K,0}(\eta).
\end{equation*} 
Thus the functional is continuous on every Fr\'echet space $Z_l^{Q-p,\, n-k-(Q-p)}(G;K)$.
By the universal property of the inductive-limit topology, it is continuous on $Z_l^{Q-p,\, n-k-(Q-p)}(G) \cap \Omega_c^{n-k}(G)$
endowed with its LF topology.
By Corollary \ref{cor: buona def pullback come integrale sulle spectral complexes},
the functional vanishes on $B_r^{Q-p,\, n-k-(Q-p)}(G) \cap \Omega_c^{n-k}(G)$. Therefore, by the universal property of the quotient topology, it descends to a continuous linear functional on $\mathscr D_{l,r}^{Q-p,\, n-k-(Q-p)}(G)$.
\end{proof}

We can now state the main result of this section.

\begin{theorem} \label{thm: pullback has a map of complexes}
Let $G$ be the Engel group and let $\varphi:G\longrightarrow G$ belong to $W_{\mathrm{loc}}^{1,q}(G,G)$ with $q>7$. For every admissible quadruple $(p,k,r,l)$, define 
\begin{equation*} 
\phi_{r,l}^{p,k-p}: E_{r,l}^{p,k-p}(G) \longrightarrow \mathscr C_{r,l}^{p,k-p}(G) 
\end{equation*}
by 
\begin{equation*} 
\phi_{r,l}^{p,k-p}([\alpha]) := \phi_{[\alpha]}, 
\end{equation*}
where $\phi_{[\alpha]}$ is the current defined in \eqref{eq: def del pullback come corrente}. The collection of maps $\{\phi_{r,l}^{p,k-p}\}$ defines a morphism from each spectral complex associated with the Engel group to the corresponding complex of spectral currents. More precisely, whenever the indices are admissible,
\begin{equation}\label{eq: pullback morphism spectral currents} 
\phi_{r,l}^{p,k-p} \circ \Delta_l = \partial_l \circ \phi_{l,m}^{p-l,\, k-1-(p-l)}.
\end{equation} 
\end{theorem}

\begin{remark}
    More explicitly, we have the two following morphisms

\[\begin{tikzcd}
	{Z_1^{0,0}(G)} & {E_{3,1}^{1,0}(G)} & {E_{2,3}^{4,-2}(G)} & {E_{1,2}^{6,-3}(G)} & {\Omega^{7,-3}(G)} \\
	{D'(G)} & {\mathscr{C}_{3,1}^{1,0}(G)} & {\mathscr{C}_{2,3}^{4,-2}(G)} & {\mathscr{C}_{1,2}^{6,-3}(G)} & {D'(G) \otimes \bigwedge^4\mathfrak{g}^\ast}
	\arrow["{\Delta_1}", from=1-1, to=1-2]
	\arrow["{\varphi_P^*}", from=1-1, to=2-1]
	\arrow["{\Delta_3}", from=1-2, to=1-3]
	\arrow["{\phi_{3,1}^{1,0}}", from=1-2, to=2-2]
	\arrow["{\Delta_2}", from=1-3, to=1-4]
	\arrow["{\phi_{2,3}^{4,-2}}", from=1-3, to=2-3]
	\arrow["{\Delta_1}", from=1-4, to=1-5]
	\arrow["{\phi_{1,2}^{6,-3}}", from=1-4, to=2-4]
	\arrow["{\varphi_P^*}", from=1-5, to=2-5]
	\arrow["{\partial_1}"', from=2-1, to=2-2]
	\arrow["{\partial_3}"', from=2-2, to=2-3]
	\arrow["{\partial_2}"', from=2-3, to=2-4]
	\arrow["{\partial_1}"', from=2-4, to=2-5]
\end{tikzcd}\]

\[\begin{tikzcd}
	{Z_1^{0,0}(G)} & {E_{2,1}^{1,0}(G)} & {E_{3,2}^{3,-1}(G)} & {E_{1,3}^{6,-3}(G)} & {\Omega^{7,-3}(G)} \\
	{D'(G)} & {\mathscr{C}_{2,1}^{1,0}(G)} & {\mathscr{C}_{3,2}^{3,-1}(G)} & {\mathscr{C}_{1,3}^{6,-3}(G)} & {D'(G) \otimes \bigwedge^4\mathfrak{g}^\ast}
	\arrow["{\Delta_1}", from=1-1, to=1-2]
	\arrow["{\varphi_P^*}", from=1-1, to=2-1]
	\arrow["{\Delta_2}", from=1-2, to=1-3]
	\arrow["{\phi_{2,1}^{1,0}}", from=1-2, to=2-2]
	\arrow["{\Delta_3}", from=1-3, to=1-4]
	\arrow["{\phi_{3,2}^{3,-1}}", from=1-3, to=2-3]
	\arrow["{\Delta_1}", from=1-4, to=1-5]
	\arrow["{\phi_{1,3}^{6,-3}}", from=1-4, to=2-4]
	\arrow["{\varphi_P^*}", from=1-5, to=2-5]
	\arrow["{\partial_1}"', from=2-1, to=2-2]
	\arrow["{\partial_2}"', from=2-2, to=2-3]
	\arrow["{\partial_3}"', from=2-3, to=2-4]
	\arrow["{\partial_1}"', from=2-4, to=2-5]
\end{tikzcd}\]

\end{remark}

\begin{proof}[Proof of Theorem \ref{thm: pullback has a map of complexes}]
Corollary \ref{cor: buona def pullback come integrale sulle spectral complexes} and Lemma \ref{lemma: continuitá pullback come corrente} show that every map $\phi_{r,l}^{p,k-p}$ is well-defined. 
It remains to prove compatibility with the differentials. Let  $[\alpha] \in E_{l,m}^{p-l,\, k-1-(p-l)}(G)$  and let $[\eta] \in \mathscr D_{l,r}^{Q-p,\, n-k-(Q-p)}(G)$. We have to show that 
\begin{equation*}
\left\langle \phi_{r,l}^{p,k-p} \bigl(\Delta_l[\alpha]\bigr), [\eta] \right\rangle  = \left\langle \partial_l \left( \phi_{l,m}^{p-l,\, k-1-(p-l)}([\alpha]) \right), [\eta] \right\rangle.  \end{equation*}
By definition, the left-hand side is \begin{equation*}
\int_G \varphi_P^* \bigl(\Delta_l[\alpha]\bigr) \wedge \eta. \end{equation*} 
On the other hand, since $\phi_{l,m}^{p-l,k-1-(p-l)}([\alpha])$ is a current of total degree $k-1$, the definition of the boundary operator gives 
\begin{equation*}
\left\langle \partial_l \left( \phi_{l,m}^{p-l,\, k-1-(p-l)}([\alpha]) \right), [\eta] \right\rangle = (-1)^k \int_G \varphi_P^*\alpha \wedge \Delta_l[\eta]. 
\end{equation*}
Therefore, it is enough to prove \begin{equation}\label{eq: integrali e pullback} 
\int_G \varphi_P^* \bigl(\Delta_l[\alpha]\bigr) \wedge \eta = (-1)^k \int_G \varphi_P^*\alpha \wedge \Delta_l[\eta]. 
\end{equation}
Choose a representative $\alpha \in [\alpha]$. By Proposition \ref{prop: differentials Delta_r}, there exist forms  $ z_{p-l+j}\in \Omega^{p-l+j, k-1 -p+l-j}(G)$ for $1\le j\le l-1$ such that
    \begin{align*}
        \int_G \varphi_P^*\Delta_l([\alpha]) \wedge \eta =  \int_G \varphi_P^*d \big(\alpha-\sum_{j=1}^{l-1}z_{p-j}\big) \wedge \eta
    \end{align*}
    and $\zeta_{Q-p+j}\in \Omega^{Q-p+j, n-k -Q+p-j}(G)$ for $1\le j\le l-1$ such that
    \begin{align*}
        \int_G \varphi_P^*\alpha \wedge \Delta_l([\eta]) = \int_G \varphi_P^*\alpha \wedge d\big(\eta - \sum_{j=1}^{l-1}\zeta_{Q-p+l-j}\big) .
    \end{align*} 
    Since the Pansu pullback preserves homogeneous weight components, all the terms of total weight larger than $Q$ vanish. Consequently, \begin{align*}
       \int_G \varphi_P^*\Delta_l([\alpha]) \wedge \eta &= \int_G \varphi_P^*d \big(\alpha-\sum_{j=1}^{l-1}z_{p-j}\big) \wedge \big(\eta - \sum_{j=1}^{l-1}\zeta_{Q-p+l-j}\big)\,,\\
       \int_G \varphi_P^*\alpha \wedge \Delta_l([\eta])  &=\int_G \varphi_P^*\big(\alpha-\sum_{j=1}^{l-1}z_{p-j}\big) \wedge d\big(\eta - \sum_{j=1}^{l-1}\zeta_{Q-p+l-j}\big)\,.
    \end{align*}
    Finally, the conclusion follows by applying Theorem \ref{th: KMX 1} to the forms $\alpha-\sum_{j=1}^{l-1}z_{p-j}$ and $\eta - \sum_{j=1}^{l-1}\zeta_{Q-p+l-j}$.
\end{proof}

\bibliographystyle{amsplain}
\bibliography{bibliography}
\end{document}